\documentclass[12pt]{amsart}

\usepackage{amssymb,latexsym}
\usepackage{amsmath,amsfonts,amsthm,amscd,amsxtra,enumerate}
\usepackage{mathtools}
\usepackage[all]{xy}
\usepackage{amsmath,latexsym,amssymb, enumerate, amsthm, epsfig, hyperref} 

\usepackage[margin=1in]{geometry}
\usepackage{epsfig}
\usepackage{young}
\usepackage[vcentermath]{youngtab}
\usepackage{young}
\usepackage{ytableau}

\usepackage[pagewise]{lineno}
\usepackage{fancyhdr}

\newtheorem{theorem}{Theorem}[section]
\newtheorem{lemma}[theorem]{Lemma}
\newtheorem{corollary}[theorem]{Corollary}
\newtheorem{definition}[theorem]{Definition}
\newtheorem{example}[theorem]{Example}
\newtheorem{remark}[theorem]{Remark}
\newtheorem{proposition}[theorem]{Proposition}

\newcommand{\p}{\mathbf{p}}

\newcommand{\bs}{\boldsymbol}

\DeclarePairedDelimiter\floor{\lfloor}{\rfloor}

\begin{document}

\title{\textbf{Spinor Structures on Free Resolutions of Codimension Four Gorenstein ideals}}

\author[Ela Celikbas]{Ela Celikbas}
\address{Department of Mathematics, West Virginia University, Morgantown, WV 26506, USA.}
\email{ela.celikbas@math.wvu.edu}

\author[Jai Laxmi]{Jai Laxmi}
\address{School of Mathematics, Tata Institute of Fundamental Research,
Mumbai-400005, India.}
\email{laxmiuohyd@gmail.com}
\email{jailaxmi@math.tifr.res.in}

\author{Jerzy Weyman}
\address{Department of Mathematics, University of Connecticut, Storrs, CT 06269, USA.}
\email{jerzy.weyman@uconn.edu}
\address{Instytut Matematyki, Uniwersytet Jagiello\'{n}ski, Krak\'{o}w 30-348, Poland.} 
\email{jerzy.weyman@uj.edu.pl}

\date{\today}


\begin{abstract}
We analyze the structure of spinor coordinates on resolutions of Gorenstein ideals of codimension four. As an application we produce a family of such ideals with seven 
generators which are not specializations of Kustin-Miller model.	
\end{abstract}

\maketitle

\section{Introduction}
\label{intro}
In this paper we investigate spinor structures on free resolutions of Gorenstein ideals of codimension 4.
Such structures were first described in \cite{R}. 

Let $R$ be a Gorenstein local ring in which $2$ is a unit and let $I\subset R$ be a Gorenstein ideal of codimension 4.  Let 
		\begin{equation}\label{eq1}
		\mathbb{F}: 0\rightarrow{}F_4\xrightarrow{{\bs d}_4} F_3 \xrightarrow{{\bs d}_3} F_2\xrightarrow{{\bs d}_2} F_1\xrightarrow{{\bs d}_1} R
		\end{equation}
		be a minimal free resolution of $R/I$. For the minimal free resolution $\mathbb F$ of a Gorenstein ideal $I$ of codimension 4, 
there is a quadratic form on $F_2$ of $\mathbb F$; see \cite[Theorem 2.4]{KM2}. It is also shown that this quadratic form on $F_2$ can be reduced to a hyperbolic form; see \cite[Theorem 5.3]{KM2}.

In this setting, we show in our main result (see Theorem \ref{mainthm}) that there is a spinor structure on $\mathbb F$. We give an explicit relation between spinor coordinates of $\text{Im}(d_3)$ and the Buchsbaum-Eisenbud multipliers. In particular, the spinor coordinates are square roots of some special Buchsbaum-Eisenbud multipliers; see Remark \ref{rmkmainthm}.

We calculate the spinor coordinates for some well-known examples of Gorenstein ideals of codimension 4 with few generators; see section~\ref{sec:examples}. For ideals with $7$ generators, Kustin and Miller constructed a family of ideals associated to a $3\times 4$ matrix, a $4$-vector, and a variable. This family is also known as the Kustin-Miller Model (KMM); see \cite{KM1,KM3}. We discuss generic doubling of almost complete intersection of codimension 3 which leads to a specialization of KMM; see Subsection~\ref{genericdoublingaci}. 

In \cite[page 29]{R}, Reid asks if every case of $7\times 12$ resolution is the known KMM. We construct a new family in Subsection \ref{newfamily}: a ``doubling" of perfect ideal with $5$ generators, of Cohen-Macaulay type $2$, described in \cite{CLWK}. As an application, using spinor coordinates, we show that this new family is not a specialization of Kustin-Miller family; see Theorem~\ref{notspofKM}.  This answers Reid's question.

Our calculations uncover a nice structure of Buchsbaum-Eisenbud multipliers and also reveal an interesting pattern.  Assume we look at a resolution of a Gorenstein ideal $I$ of codimension 4 in a local ring $(R, \mathfrak{m})$. In all the examples we know, the spinor coordinates belong to the ideal $I$; see Remark~\ref{rmk3}. Furthermore, for all known examples of ideals $I$ with 6, 7, and 8 generators, some spinor coordinates are not in ${\mathfrak m} I$, and hence they can be taken as minimal generators of $I$.
However, for 9 or more generators, we find an example when all spinor coordinates are in ${\mathfrak m}I$. This suggests that Gorenstein ideals of codimension 4 with up to 8 generators are easier to classify than those with more than 8 generators; see Remark~\ref{finalrmk}.

The paper is organized as follows.  As the intended audience are commutative algebraists, we include an extended section \ref{sec:rep} on representations of general linear groups and special orthogonal groups.
In Subsection \ref{sec:ortho}, while working with orthogonal spin groups, we first deal with characteristic zero case, and then indicate which results stay true in arbitrary characteristic different from $2$. We also include subsection \ref{isotropic} to explain the connection of spinor structure to the geometry of isotropic Grassmannians.
In Section \ref{sec:background} we recall the Buchsbaum-Eisenbud First Structure Theorem  for finite  free resolutions and the results of Kustin-Miller on the resolutions of Gorenstein ideals of codimension 4.

In Section \ref{sec:spinor} we prove the existence of spinor structures on resolutions of Gorenstein ideals of codimension 4. We also apply results from section \ref{sec:p} to explicitly calculate the relation between the Buchsbaum-Eisenbud multipliers and the spinor coordinates.

Section \ref{sec:examples} contains the computations of the spinor coordinates for complete intersections, for hyperplane sections of codimension 3 Gorenstein ideals, and for KMM with 7 generators.

In Section \ref{sec:doublingsci}  we analyze the resolutions of ``doublings" of almost complete intersection ideals of codimension 3.  We show that these ideals are closely related
to the KMM model. Finally, in Section \ref{(1,5,6,2)D}, we show that the Gorenstein ideals of codimension 4 that are doublings of perfect ideals of codimension 3 with 5 generators of Cohen-Macaulay type 2 found in \cite{CLWK}  are not specializations of the KMM.

\section{Background in representation theory}\label{sec:rep}

\subsection{Representation Theory of $\mathbf {GL({V})}$}

Let $V$ be a vector space over a field $K$ (or a free module over a ring $R$).
 We will use the following notation for the representations of the group $GL_n(V)$. For the dominant integral weight $(a_1,\cdots,a_n)$ where $a_i\in \mathbb{Z}$ and  $a_1\geq a_2\geq \cdots\geq a_n$,  $S_{(a_1,\cdots,a_n)}{V}$ denotes the corresponding Schur module. We denote the Lie algebras of $GL(V)$ and $SL(V)$ as ${\underline gl}(V)$ and ${\underline{ sl}}(V)$, respectively.
 The main reference we will be using is the book \cite{FH}, lecture $6$.

\subsection{Preliminaries on representations of the spin groups.}\label{sec:ortho}

In this section we collect the material involving representation theory of the spin group. We will work over a field $K$ of characteristic different from $2$.
However, for the convenience of the reader, we first recall the basic facts over a field of characteristic zero, and then indicate what needs to be modified in finite characteristics different than $2$. 

\subsubsection{Representations of the spin group over an algebraically closed field of characteristic zero}
Most of the material in this section can be found in the book of Fulton and Harris \cite{FH}, in lectures 18--20.
Other references are \cite[Chapters 2,3,6]{GW}, \cite[Section 2.15]{Jen}, and \cite[Chapter 2]{DP}.
Let $K$ be an algebraically closed field of characteristic zero and let $V$ be an orthogonal space of dimension $2m$ over $K$. We put the symmetric bilinear map $Q:V\otimes V\rightarrow K$ in the hyperbolic form. More precisely, let $W$ be an isotropic space in $V$ of dimension $m$. We can identify $V$ with $W\oplus W^*$ and the symmetric bilinear map $Q$ with the duality
$$Q:W\otimes W^*\rightarrow K,$$
 also requiring $W$ and $W^*$ being isotropic.  

Throughout we deal with the representations of the special orthogonal Lie algebra ${\underline so}(V)$, as it is well known that the categories of rational representations of the spin group Spin$(2m)$ and of ${\underline so}(V)$ are equivalent.
The maximal toral subalgebra in the Lie algebra ${\underline so}(V)$ is the maximal toral subalgebra of diagonal matrices in ${\underline gl}(W)$. It consists of matrices
$$\left(\begin{matrix} A&0\\0&-A\end{matrix}\right)$$
where $A$ is an $m\times m$ diagonal matrix. We denote the basis of $V$ as follows. Vectors $\lbrace e_1,\ldots ,e_m\rbrace$ form a basis of $W$, and $\lbrace e_{-1}=e_1^*,\ldots ,e_{-m}=e_m^*\rbrace$ form the dual basis in $W^*$. Their weights are $\varepsilon_i$ and $-\varepsilon_i$ for $1\le i\le m$, respectively.

Since  the symmetric bilinear map $Q$ is in standard form in any characteristics different from $2$, we can use representation of Spin($2m$) as well. For the representation of the spin group, we use the following notation.  
In this case, a maximal torus $T$ of Spin($2m$) and the Lie algebra of $T$, denoted $\mathfrak{t}$, are 
$$T=\{{\rm diag}[x_1,\cdots,x_m,x_{m}^{-1},\cdots,x_1^{-1}]: x_i\in K\setminus \{0\}\},$$ 
$$\mathfrak{t}=\{{\rm diag}[a_1,\cdots,a_m,-a_{m},\cdots,-a_1]: a_i\in K\}.$$
For $i=1,\ldots,m$, define $\langle {\varepsilon_i},{D}\rangle=a_i$, where ${D}={\rm diag}[a_1,\cdots,a_m,-a_{m},\cdots,-a_1]$ is in $ \mathfrak{t}$. Then $\{{\varepsilon_1},\ldots,{\varepsilon_{m}}\}$ is a basis for $\mathfrak{t}^*$.
All representations of Spin$(2m)$ restrict to $T$ so they decompose to weights with respect to $T$. The simple roots are $\varepsilon_i-\varepsilon_{i+1}$ for $i<m$ with $\varepsilon_{m-1}+\varepsilon_m$.


Let us recall that when $K$ is algebraically closed of characteristic zero, then  irreducible representations of ${\underline so}(V)$ are parametrized by dominant integral weights 
$$\lambda =\sum_{i-1}^m\lambda_i\omega_i,$$
where $${\omega}_{i}={\varepsilon}_1+\cdots+{\varepsilon}_{i}$$ for $1\leq i\leq m-2$, $ { \omega_{m-1}}=\frac{1}{2}({\varepsilon}_1+\cdots+{\varepsilon}_{m-1}+{\varepsilon}_{m})$, and ${\omega}_m=\frac{1}{2}({\varepsilon}_1+\cdots+{\varepsilon}_{m-1}-{\varepsilon}_{m})$ are so-called fundamental weights, and $\lambda_i\in{\bf Z}_{\ge 0}$. 

We denote $V(\lambda)$ the irreducible representation corresponding to the highest weight $\lambda$.
The fundamental representations of Spin$(2m)$ are $V(\omega_i)=\bigwedge^i V$ for $1\le i\le m-2$, and the fundamental representations are the half-spinor representations for $i=m-1$ and $i=m$.
To define them, we need a Clifford algebra 
$$C(V, Q)=T(V)/I(V),$$
where $T(V)$ is a tensor algebra of $V$ and $I(V)$ is the two-sided ideal in $T(V)$ generated by the elements
$$v_1\otimes v_2+v_2\otimes v_1-2Q(v_1, v_2)$$
for $v_1, v_2\in V$.
Note that since the ideal $I(V)$ has generators with components in the $0$-th and $2$-nd graded component of $T(V)$, the Clifford algebra decomposes to its even part $C(V)_+$ and its odd part $C(V)_-$. Additively we have decompositions
\begin{align*}C(V)_+&=\oplus_{i\ \rm{even}}\bigwedge^i V, \\
 C(V)_-&=\oplus_{i\ \rm{odd}}\bigwedge^iV.
 \end{align*}
Let $f=e_{1}\wedge\ldots \wedge e_{m}$. 

We have the following result (see \cite{FH}, lecture 20 for more details, note that our convention interchanges $W$ and $W^*$).

\begin{proposition} The left ideal
$$S= C(Q)^.f$$
is additively isomorphic to the exterior algebra $\bigwedge^\bullet W^*$.
It is therefore a representation of ${\underline so}(V)$. It decomposes to even and odd parts $S_+:=\bigwedge^{even}W^*$ and $S_-:=\bigwedge^{odd}W^*$.
$S$ is called a Clifford module, and $S_+$ and $S_-$ are called half-spinor modules.
\end{proposition}
We also have
\begin{align*}V(\omega_{m-1})&=S_+ \\ 
V(\omega_m )&=S_-.
\end{align*}

Both half-spinor representations have dimension $2^{m-1}$. 
Let $\mathcal{L}\subset [1,m]$ be a subset, let $\mathcal{L}^c$ be its complement. We  denote by $u_\mathcal{L}$ a coset of the tensor $\wedge_{i\in \mathcal{L}} w_{-i}$. This is a weight vector of weight
$-{1\over 2} (\sum_{i\in \mathcal{L}}\varepsilon_i +\sum_{i\in \mathcal{L}^c}\varepsilon_i)$.

Note that Clifford algebra is then identified with ${\rm End}_{C(Q)}(S)$ and the spin group appears as certain subset of invertible elements of $C(Q)$. We do not need this description, so
we refer the reader to \cite{FH}, lecture $20$.

For the convenience of the reader we describe the action of the Lie algebra ${\underline so}(V)$ on half-spinor representations. Strictly speaking, it will not be needed but it explains weight decompositions of half-spinor representations.

For ${a},{b}\in {V}$, define $R_{{a},{b}}\in {\rm End}({V})$ as $R_{{a},{b}}{v}=Q({b},{v}) {a}-Q ({a},{v}) {b}$. By \cite[Section 2.4]{DP}, $R_{{a},{b}}$ spans ${\underline so}(V)$ for ${a},{b}\in {V}$. Then $R_{{e}_i,{e}_j}={e}_{-i,j}-{e}_{-j,i}$ where ${e}_{i,j}$ be an elementary transformation on ${V}$ that carries ${e}_i$ to ${e}_j$ and others to 0.  

For ${y}^*\in {W}^*$, the exterior product $\bs \epsilon({y}^*)$ and the interior product operator $\mathfrak{i}({y})$ on $\bigwedge {W}$ are defined as $\bs \epsilon({ y}^*){x}^*={y}^*\wedge {x}^*$ and  
$$\mathfrak{i}({y})({y}_1^*\wedge \cdots \wedge  {y}_k^*)=\sum\limits_{j=1}^{k}(-1)^{j-1}Q( {y},{y}_j^*) {y}_1^*\wedge\cdots\wedge \widehat{{y}_j^*}\wedge \cdots \wedge {y}_k^*,$$
where ${y}_i^*\in {W}^*$, ${x}^*\in \bigwedge {W}^*$ and $\widehat{{y}_j^*}$ means to omit ${y}_j^*$. 

Define linear maps $\bs \gamma: {V}\rightarrow \text{End}(\bigwedge {W}^*)$ as $\bs \gamma({y}+{y}^*)=\mathfrak{i}({y})+\bs \epsilon({y}^*)$ for ${y}\in W$ and ${y}^*\in {W}^*$, and $\bs \varphi: {\underline so}(V)\rightarrow {\rm Cliff}_2({V},Q)$ as $\bs \varphi(R_{{a},{b}})=\frac{1}{2}[\bs \gamma({a}),\bs \gamma({b})]$ for ${a},{b}\in {V}$ where $[\bs \gamma({a}),\bs \gamma({b})]=\bs \gamma({a})\bs \gamma({b})-\bs \gamma({b})\bs \gamma({a})$. By \cite[Chapter 2]{DP}, $\bs\varphi$ is injective, and the Lie algebra of Spin(${V}$) is $\bs \varphi({\underline so}(V))$.

Let us also look at other exterior powers of $V$. We have
$$\bigwedge^{m-1}V=V(\omega_{m-1}+\omega_m),$$
$$\bigwedge^m V=V(2\omega_{m-1})\oplus V(2\omega_m).$$

To see the decomposition in the second formula, we proceed as follows.
Let ${\tilde Q}:V\rightarrow V^*$ be an ${\underline so}(V)$-equivariant isomorphism defined by the formula
$${\tilde Q}(v_1)(v_2):=Q(v_1, v_2).$$
This isomorphism defines a similar ${\underline so}(V)$-equivariant isomorphism
$$\bigwedge^m{\tilde Q}:\bigwedge^m V\rightarrow\bigwedge^m V^*.$$
We also have an ${\underline sl}(V)$-equivariant isomorphism
$$\phi:\bigwedge^m V^*\rightarrow\bigwedge^m V,$$
using $e_1^*\wedge\ldots\wedge e_m^*\wedge e_1\wedge\ldots\wedge e_m$  as a volume form.
We define an ${\underline so}(V)$-equivariant isomorphism
$$\tau=\phi\circ (\bigwedge^m {\tilde Q}):\bigwedge^m V\rightarrow\bigwedge^m V.$$
One proves easily that $\tau^2=1$. The representation $V(2\omega_{m-1})$ can be identified with the $1$-eigenspace of $\tau$ and
$V(2\omega_{m})$ can be identified with the $-1$-eigenspace of $\tau$. Thus the operators ${1\over 2}(\tau-1)$ and ${1\over 2}(\tau+1)$
are the projections on both direct summands.

Let us also mention the tensor product decompositions that will be useful. They go back at least to 1967 Cartan's book \cite{C}, but, for our purposes, we refer to \cite{A} and \cite{DP}.

\begin{proposition}\label{prop:tensor} \cite[Theorem 4.6]{A} Let $K$ be an algebraically closed field of characteristic zero.\\
\begin{enumerate}
\item $$\bigwedge^2 V(\omega_1)=V(\omega_2)$$
 \item $$S_2 V (\omega_1)=V(2\omega_1)\oplus K$$
\item $$\bigwedge^2 V(\omega_{m-1})=\oplus_i V(\omega_{m-2-4i})$$
\item $$S_2 V (\omega_{m-1})=V(2\omega_{m-1})\oplus\oplus_i V(\omega_{m-4i})$$
\item $$\bigwedge^2 V(\omega_m)=\oplus_i V(\omega_{m-2-4i})$$
\item $$S_2 V (\omega_m)=V(2\omega_m)\oplus\oplus_i V(\omega_{m-4i})$$
\end{enumerate}
with the convention that $V(\omega_0)=K$.

\end{proposition}

\subsubsection{Arbitrary characteristic} 

Let $K$ be algebraically closed field of characteristic different than $2$. We work with the symmetric bilinear form $Q$ which is already in hyperbolic form, so we can set again
$V=W\oplus W^*$ and use the basis $\lbrace e_1,\ldots ,e_m, e_{-1},\ldots ,e_{-m}\rbrace$.
The  representations $V(\omega_i)$ for $1\le i\le m-2$ can be constructed over $K$. They might not be irreducible but this will not be relevant.
The construction of half-spinor representations also can be carried out over $K$.
Among the tensor product decompositions that were listed the most important are the ones for the symmetric powers of half-spinor representations; see \cite{DP}.

The decomposition of $\bigwedge^m V$ into two summands of equal dimension as described in previous section is also true over $K$. The summands might not be irreducible, but this will not be relevant.

The constructions of fundamental representations $\bigwedge^i V$, $V(\omega_{m-1})$, and $V(\omega_m)$ are also valid over any commutative ring $R$ such that $2$ is invertible
in $R$. One assumes that $V$ is a free orthogonal module, i.e. $V$ has a symmetric bilinear form which in some basis is hyperbolic. Let us state the definition.

\begin{definition}\label{def1}
		Let $V$ be a finitely generated free $R$-module. A bilinear map $Q : V\otimes_R V\rightarrow R$  is called a non-degenerate pairing if it is a symmetric bilinear map such that the induced map $ {\tilde Q}:V\rightarrow V^*$ defined by $f\mapsto  Q(-,f)$ is an isomorphism. A non-degenerate pairing  is
		in standard form if ${\rm rank}(V)$ is even and we can write it as a direct sum of hyperbolic two-dimensional pairings with matrices of the form
		$$\begin{bmatrix}0&1\\ 1&0\end{bmatrix}.$$
		In this case, $V$ is also called an even orthogonal $R$-module and the corresponding basis of $V$ is called a hyperbolic basis.
\end{definition}

\subsubsection{Isotropic Grassmannians} \label{isotropic}

Let $P_m$ and $P_{m-1}$ denote maximal parabolic subgroups corresponding to the simple roots $\varepsilon_{m-1}-\varepsilon_{m}$ and $\varepsilon_{m-1}+\varepsilon_{m}$, respectively.

The spinor coordinates are associated to the homogeneous spaces Spin($2m)/P_{m-1}$ and Spin$(2m)/P_m$ which are connected components of the isotropic Grassmannian IGrass$(m, V)$. These spaces are closed subvarieties of the projective spaces ${\bf P}(V(\omega_{m-1}))$ and ${\bf P}(V(\omega_m))$, respectively.

The Pl\"ucker embeddings are the doubles of these fundamental embeddings. The equations defining the subvariety Spin$(2m)/P_{m-1}$ inside of ${\bf P}(V(\omega_{m-1}))$
(respectively Spin$(2m)/P_{m}$ inside of ${\bf P}(V(\omega_{m}))$) are quadratic. They are well known in commutative algebra, see  \cite{CM}.

In order to make everything explicit, let us choose a hyperbolic basis in $V$: $$\lbrace e_1,\ldots ,e_m, e_{-1},\ldots ,e_{-m}\rbrace.$$  Consider a subspace $U$ in IGrass$(m, V)$ whose Pl\"ucker coordinate corresponding to
$e_1,\ldots ,e_m$ is nonzero (this contains a choice of connected component of IGrass$(m, V)$ in which $U$ is contained). Then the subspace $U$ has a unique basis whose expansions in our hyperbolic basis give rows of a matrix
$$M=\left(\begin{matrix}J&X\end{matrix}\right)=\left(\begin{matrix} 0&0&\ldots&0&1&x_{1,1}&x_{1,2}&\ldots &x_{1,m}\\
0&0&\ldots&1&0&x_{2,1}&x_{2,2}&\ldots &x_{2,m}\\
\ldots&\ldots&\ldots&\ldots&\ldots&\ldots&\ldots&\ldots&\ldots\\
1&0&\ldots&0&0&x_{m,1}&x_{m,2}&\ldots &x_{m,m}
\end{matrix}\right)$$

The subspace $U$ is isotropic which makes the matrix $X$ skew-symmetric. So the big cell $Z$  in Spin$(2m)/P_{m-1}$ is isomorphic to the space of skew-symmetric matrices.
The Pl\"ucker coordinates restrict on $Z$ to the maximal minors of the matrix $M$, and the spinor coordinates restrict to Pfaffians of all sizes of all submatrices of $X$ which are themselves skewsymmetric. Each such Pfaffian can be described by even number of columns of $X$ (taking also the same rows), and it is the spinor coordinate (i.e. weight vector in $V(\omega_{m-1})$) whose weight has minus signs precisely at these places. This includes identity which corresponds to empty subset of columns.

So the quadratic relations defining our homogeneous space are the quadratic identities among Pfaffians of all sizes of the matrix $X$.
All such relations are written in \cite{CM}. See also \cite{KU} section $3$ for a similar set of relations.

\begin{example} Let us write explicitly the case of $n=5$ where we have $10$ quadratic relations, going back to Cartan; see \cite{C}. Let $\mathcal{I}\subset [1,m]$ be a subset of even cardinality. We write $pf(\mathcal{I})$ for the Pfaffian of a submatrix of $X$
on the rows and columns from $\mathcal{I}$. The Cartan equations are
$$pf(\emptyset)pf(1234)-pf(12)pf(34)+pf(13)pf(24)-pf(14)pf(23)$$
$$pf(12)pf(1345)-pf(13)pf(1245)+pf(14)pf(1235)-pf(15)pf(1234)$$
and eight others which we get by permuting the numbers $1,2,3,4,5$.
\end{example}

This whole construction is characteristic free.
In degree $2$ we see that $V(2\omega_{m-1})$ is a factor of $S_2(V(\omega_{m-1}))$ by the span of these quadratic equations. This is an analogue of the decomposition of $S_2 (V(\omega_{m-1}))$ given in \ref{prop:tensor}  in positive characteristic.

There is another set of interesting relations related to this situation. Since Pl\"ucker embedding is a double of a fundamental embedding, each minor of $X$ has a quadratic expression in terms of its Pfaffians. Such expressions are also known (see \cite{Kn} and references there).

The equivariant map $\p$ constructed in the next section allows to see these relations in the context of representations of the spin group (see Remark \ref{rmk3} following the proof of the main theorem).

\subsection{Certain Spin($2m$)-equivariant map $\p$ and its properties}\label{sec:p}

Let $V$ be an orthogonal space of rank $2m$ over algebraically closed field $K$ of characteristics different from 2. Our goal in this section is to describe certain equivariant map $$\p:S_2(V(\omega_m))\rightarrow \bigwedge \limits^m V$$ explicitly.  If $K$ has characteristic zero, then, by formula \ref{prop:tensor}, we have a unique such Spin($2m$)-equivariant map $\p$ up to scalar.  Over fields $K$ of characteristic different from $2$ one can check that formulas we write down below define an equivariant map, so we will just use it.

\begin{remark} The  map $\p$ will be very important in our application as it will give polynomial formula expressing arbitrary Buchsbaum-Eisenbud multipliers by quadratic expressions involving spinor coordinates.
\end{remark}

Before we start we need some notation.
The signature of a permutation of the set $[1,m]$, denoted by sgn, is a multiplicative map from the group of permutations $S_m$ to $\pm 1$. Permutations with signature +$1$ are even and those with sign -$1$ are odd. Also $\mathcal{L}^c$ denotes the complement of a subset $\mathcal{L}$ of $[1, m]$.


\begin{lemma}\label{lem:mapp} {
		Set $q=\floor*{\frac{m}{2}}$. There is an equivariant map	$\p:S_2({V}(\omega_m))\rightarrow \bigwedge \limits^m {V}$ which is defined as
\begin{align}\label{eqn}
		\p({u}_{\mathcal{J}_{2k}}{u}_{ \phi})=\dfrac{1}{2^{\ell(\mathcal{J}_{2k})-1}}\sum\limits_{\mathcal{L}\subset \mathcal{J}_{2k},\ell(\mathcal{J}_{2k})=2\ell(\mathcal{L})}{\rm sgn}(\mathcal{J}_{2k},\mathcal{L}) {e}_{-\mathcal{L}}\wedge {e}_{\mathcal{J}_{2k}^{c}}\wedge {e}_\mathcal{L}
		\end{align}
such that $\p({u}_{ \phi}{u}_{\phi})={e}_1\wedge {e}_2\wedge\cdots \wedge {e}_m$. Here $\mathcal{J}_{2k}=\{\gamma_1,\ldots,\gamma_{2k}\}$ with $1\leq \gamma_1< \cdots < \gamma_{2k}\leq m$, $1\leq k\leq q$,  ${e}_{-\mathcal{L}}=\bigwedge\limits_{i\in \mathcal{L}} {e}_{-l}$, ${u}_{\mathcal{J} _{2k}}={e}_{-\gamma_1}\wedge {e}_{-\gamma_2}\wedge\cdots\wedge {e}_{-\gamma_{2k}}$, $1\leq k\leq q$, ${e}_\mathcal{L}=\bigwedge\limits_{i\in \mathcal{L}} {e}_{l}$, {\rm sgn}$(\mathcal{J}_{2k},\mathcal{L})$ is the signature of permutations of $\mathcal{J}_{2k}$, and $\ell(\mathcal{J})$ is the length of any indexing set $\mathcal{J}\subset [1,m]$.}

\end{lemma}

\begin{proof}
	We prove formula (\ref{eqn}) by reverse induction on $k$.
	For $k=q$
	$$\p({u}_{\mathcal{J}_{2q}}{u}_{ \phi})=\dfrac{1}{2^{\ell(\mathcal{J}_{2q})-1}}\sum\limits_{\mathcal{L}\subset \mathcal{J}_{2q},\ell(\mathcal{J}_{2q})=2\ell(\mathcal{L})}{\rm sgn}(\mathcal{J}_{2q},\mathcal{L}) { e}_{-\mathcal{L}}\wedge {e}_{\mathcal{J}_{2q}^{c}}\wedge {e}_\mathcal{L}.$$
	Using the action of internal operator on $V(2\omega_m)$ we see that $$\mathfrak{i}({e}_{\gamma_i})\mathfrak{i}({e}_{\gamma_j})({u}_{\mathcal{J}_{2q}}{u}_{ \phi})=(-1)^{i+j}{u}_{\mathcal{J}_{2q}\setminus\{\gamma_i,\gamma_j\}}{u}_{\phi}.$$ 
The map $\p$ is equivariant, then, by \cite[Lemma 6.2.1]{GW},	the following diagram
$$
	\xymatrixrowsep{1.2pc} \xymatrixcolsep{2.2pc}
	\xymatrix{
		{V}(2\omega_m)\ar@{->}[r]^{\p}\ar@{->}[d]_{\mathfrak{i}({\bs e}_i)\mathfrak{i}({\bs e}_j)}& \bigwedge\limits^m {V}\ar@{->}[d]^{R_{{ e}_i,{ e}_j}}\\
		{V}(2\omega_m) \ar@{->}[r]^{\p}& \bigwedge\limits^m {V}\\
	}
	$$
commutes. Therefore  $\p({u}_{J_{2q}\setminus \{\gamma_i,\gamma_j\}}{u}_{ \phi})$ is of the form
	
	\begin{eqnarray*}
	\dfrac{1}{2^{2q-3}}\sum\limits_{\mathcal{L}\subset \mathcal{J}_{2q}\setminus \{\gamma_i,\gamma_j\},\ell(\mathcal{L})=q-1}{\rm sgn}(\mathcal{J}_{2q}\setminus\{\gamma_i,\gamma_j\},\mathcal{L}){e}_{-\mathcal{L}}\wedge  {e}_{(\mathcal{J}_{2q}\setminus \{\gamma_i,\gamma_j\})^c}\wedge {e}_{\mathcal{L}}.
	\end{eqnarray*}
	  Applying interior operator successively, one
	 gets expression for $k=1$ as
	$$\p({u}_{\{\gamma_i,\gamma_j\}}{u}_{ \phi})=\dfrac{{\rm sgn}(\{\gamma_i,\gamma_j\},\gamma_i,\gamma_j)}{2}({e}_{-\gamma_i}\wedge {e}_{\{\gamma_i,\gamma_j\}^c}\wedge {e}_{\gamma_i}+{e}_{-\gamma_j}\wedge {e}_{\{\gamma_i,\gamma_j\}^c}\wedge {e}_{\gamma_j}).$$
Again, by applying internal operator, we obtain
	$$\p({u}_{\phi}{u}_{\phi})= {e}_1\wedge \cdots \wedge {e}_{m}.$$
	By setting $\p|_{V(m-4i)}=0$ for $i\geq 1$, we get $\p: S_2({V}(\omega_m))\rightarrow \bigwedge\limits^m {V}$. 
\end{proof}

\begin{remark}\label{rmk2} 
		Let $\mathcal{L},\mathcal{M}\subset [1,m]$ of even cardinality. Set $\mathcal{L}\ominus \mathcal{M}=(\mathcal{L}\setminus \mathcal{M})\cup (\mathcal{M}\setminus \mathcal{L})$. Assume that $\mathcal{L}\ominus \mathcal{M}$ is nonempty. Note that $\mathcal{L}\ominus \mathcal{M}$ is of even cardinality.  Using Lemma \ref{lem:mapp}, one can evaluate the map $\p$ by permuting indices of the monomial ${u}_\mathcal{L}{u}_\mathcal{M}$  as
		
		\begin{align*}
		\dfrac{1}{2^{\ell(\mathcal{L}\ominus \mathcal{M})-1}}\sum\limits_{\mathcal{N}\subset \mathcal{L}\ominus \mathcal{M}, \ell(\mathcal{L}\ominus \mathcal{M})=2\ell(\mathcal{N})}{\rm sgn}(\mathcal{L}\cup \mathcal{M},\mathcal{N}) {e}_{-(\mathcal{L}\cap \mathcal{M})}\wedge {e}_{-\mathcal{N}}\wedge {e}_{\mathcal{L}^c\cap \mathcal{M}^c}\wedge {e}_\mathcal{N}.
		\end{align*} 
		Moreover, $\p({u}_\mathcal{L} {u}_\mathcal{L})={\rm sgn}(\mathcal{L},\mathcal{L}^c){e}_{-\mathcal{L}}\wedge {e}_{\mathcal{L}^c}$. 
\end{remark}

\begin{remark} The formulas for the map $\p$ remain valid over any commutative ring $R$ with $2$ invertible in $R$ and any orthogonal module $V$ in the sense of Definition \ref{def1}.
One can just apply the formula from Lemma \ref{lem:mapp} and they remain true in any hyperbolic basis of $V$.
\end{remark}

\section{Background on free resolutions}\label{sec:background}

Throughout the rest of the paper  $R$ and $S$ denote Noetherian commutative rings unless otherwise stated, $\mu(I)$ denotes the minimal number of generators of an ideal $I$ of $R$, and ${\rm I}_n$ denotes the $n\times n$ identity matrix. For a ring map $f:R\rightarrow S$, $\ker(f)$ denotes the kernel and $\text{im}(f)$ denotes the image of $f$, respectively. For an $R$-module $M$, $M^*$ denotes ${\rm Hom}_R(M,R)$. 

Buchsbaum and Eisenbud gave a structure theorem (also known as the First Structure Theorem) that describes the arithmetic structure of free resolutions as follows:

\begin{theorem}\label{th1} { \cite[Theorem 3.1  ]{BE1} (The First Structure Theorem)
		Let $R$ be a Noetherian ring and let $I$ be an ideal of $R$. Let 
		\[0\xrightarrow{} F_n\xrightarrow{{\bs d}_n} F_{n-1}\xrightarrow{{\bs d}_{n-1}}\cdots\xrightarrow{{\bs d}_3}F_2\xrightarrow{{\bs d}_2}F_1\xrightarrow{{\bs d}_1} F_0\]
		be a free $R$-resolution of $R/I$ and $r_i={\rm rank}({\bs d}_i)$. Then there exists a unique sequence of homomorphisms $\bs a_k:R\rightarrow \bigwedge\limits^{r_k} F_{k-1} $ for $1 \leq k \leq n$ such that $\bs a_n:=\bigwedge\limits^{r_n}{\bs d}_n$ and the following diagram commutes:	
	$$
	\xymatrixrowsep{1.2pc} \xymatrixcolsep{2.2pc}
	\xymatrix{
		\bigwedge\limits^{r_k} F_k \ar@{->}[r]^{\bigwedge\limits^{r_k}{\bs d}_k}\ar@{->}[d]_{\simeq} & \bigwedge\limits^{r_k} F_{k-1} \\
		\bigwedge\limits^{r_{k+1}} F_k^{*} \ar@{->}[r]^{\bs a_{k+1}^*}& R \ar@{->}[u]_{\bs a_{k}}\\
	}
	$$
	}
	
We refer to maps $\bs a_k$ in Theorem~\ref{th1} as the Buchsbaum-Eisenbud multiplier maps. Their coordinates are called the Buchsbaum-Eisenbud multipliers. 
	
\end{theorem}

The next remark reveals the structure of a minimal free resolution of $R/I$ where $R$ is a complete regular local ring and $I$ is a Gorenstein ideal of codimension four.

\begin{remark}\label{rmk1}{ {\rm \cite{KM2}}
		Let $R$ be a Gorenstein local ring in which $2$ is a unit and let $I\subset R$ be a Gorenstein ideal of codimension four with $\mu(I)=n$.  Let 
		\begin{equation}\label{eq1}
		\mathbb{F}: 0\rightarrow{}F_4\xrightarrow{{\bs d}_4} F_3 \xrightarrow{{\bs d}_3} F_2\xrightarrow{{\bs d}_2} F_1\xrightarrow{{\bs d}_1} R
		\end{equation}
		be a minimal free resolution of $R/I$. Then we have the following:
		\begin{enumerate}[\rm (a)]
			\item By Gorenstein duality, $F_{4-i}\cong F^*_{i}$.
			\item ${\rm rank}(F_1)=n$ and ${\rm rank}(F_2)=2n-2$.
			\item By \cite[Theorem 2.4]{KM2}, for a minimal resolution $\mathbb F$ of $R/I$ 
			\[\mathbb{F}: 0\rightarrow{}R\xrightarrow{{\bs d}_1^*} F_1^* \xrightarrow{{\bs d}_3} F_2\xrightarrow{{\bs d}_2} F_1\xrightarrow{{\bs d}_1} R,\]
			there exist a symmetric isomorphism ${\tilde Q}$ and an isomorphism $\rho:\mathbb{F}\rightarrow \mathbb{F}^*$  of the form: 
			\[
			\xymatrixrowsep{1.2pc} \xymatrixcolsep{2.2pc}
			\xymatrix{
				0\ar@{->}[r]^{}&R\ar@{->}[r]^{{\bs d}_1^*}\ar@{=}[d]^{}&F_1^*\ar@{->}[r]^{{\bs d}_3}\ar@{=}[d]^{}&F_2\ar@{->}[r]^{{\bs d}_2}\ar@{->}[d]^{{\tilde Q}}&F_1\ar@{->}[r]^{{\bs d}_1}& R\ar@{=}[d]^{}\\
				0\ar@{->}[r]^{}&R\ar@{->}[r]^{{\bs d}_1^*}&F_1^* \ar@{->}[r]^{{\bs d}_2^*}&F_2^*\ar@{->}[r]^{{\bs d}_3^*}& F_1\ar@{=}[u]^{}\ar@{->}[r]^{{\bs d}_1}&R \\
			}
			\]
			
			\item If $R$ is a complete regular local ring, then the dualizing matrix ${\tilde Q}$ is of the form $$\begin{bmatrix}
			0 & {\rm I_{n-1}}\\
			{\rm I_{n-1}} & 0
			\end{bmatrix},$$
			and, by part (c),  the resolution of $R/I$ is 
			 \[\mathbb{F}: 0\rightarrow{}R\xrightarrow{\bs d_1^t} F_1^* \xrightarrow{\bs {\tilde Q} \bs d_2^t} F_2\xrightarrow{\bs d_2} F_1\xrightarrow{\bs d_1} R.\]

			\item  Let $\langle \  , \  \rangle: F_1\otimes_R F_1^{*}\rightarrow R$ be the evaluation map and let $Q: F_2\otimes_R F_2\rightarrow R$ be the symmetric bilinear map induced by ${\tilde Q}$. Then  
			\begin{align*}\label{adj}
			\langle {\bs d}_2 x_2,x_3\rangle=Q(x_2,{\bs d}_3 x_3)\;\; \text{for all}\;\; x_2\in F_2, \;   x_3\in F_1^*.
			\end{align*}
			\item $\mathbb{F}$ has a multiplicative structure which makes it an associative differential graded $R$-algebra.  
			
			\item The module $F_2$ has a structure of an even orthogonal module of rank $2n-2$ according to  Definition~\ref{def1}. 
		\end{enumerate}
	}
\end{remark}

\begin{remark} Let $R$ be a regular local ring or a polynomial ring over a field and let $I\subset R$ be a Gorenstein ideal of codimension four with $\mu(I)=n$.  Let 
		\begin{equation}\label{eq1}
		\mathbb{F}: 0\rightarrow{}F_4\xrightarrow{{\bs d}_4} F_3 \xrightarrow{{\bs d}_3} F_2\xrightarrow{{\bs d}_2} F_1\xrightarrow{{\bs d}_1} R
		\end{equation}
		be a minimal free resolution of $R/I$. 
Then, by Theorem~\ref{th1}, we have $\bs a_4= \bs d_4$ and there is a map $\bs a_3:R\rightarrow \bigwedge\limits^{n-1} F_{2}$ such that the diagram commutes 
$$
	\xymatrixrowsep{1.2pc} \xymatrixcolsep{2.2pc}
	\xymatrix{
		\bigwedge\limits^{n-1} F_3 \ar@{->}[r]^{\bigwedge\limits^{n-1}{\bs d_3\;\;}}\ar@{->}[d]_{\simeq} & \bigwedge\limits^{n-1} F_{2} \\
		 F_3^{*} \ar@{->}[r]^{\bs d_4^*}& R \ar@{->}[u]_{\bs a_{3}}\\
	}
	$$


\end{remark}

%

\subsection{Generic doubling of perfect ideals}
Let $R$ be a regular local ring, let $J$ be a perfect ideal of $R$ of codimension 3, and  let $S:=R/J$. Assume $S$ is generically Gorenstein with a canonical module $\omega_S$. It is known that one can identify $\omega_S$  with an ideal of $S$ and $S/\omega_S$ is Gorenstein, \cite[Proposition 3.3.18]{BH}. 

Let $(\mathbb{F},{\bs d})$ be a minimal free resolution of $S$ over $R$. Then 
$(\mathbb{F}^{*},{\bs d}^*)$ is a minimal free resolution of $\omega_{S}$ as $J$ is perfect. 
Take the minimal generators $f_1, \ldots, f_{\ell}$ of  ${\rm Hom}_S(\omega_S,S)$, and let $\widetilde{R}:=R[\tau_1, \ldots, \tau_{\ell}]$. Now we consider the injective map $\psi: \omega_{\widetilde{S}} \rightarrow \widetilde{S}$ where $\psi=\sum_{i=1}^{\ell} \tau_if_i$, $\widetilde{S}:=\widetilde{R}/J\widetilde{R}$, and  $\omega_{\widetilde{S}}$ is a canonical module of $\widetilde{S}$. 
Then $\omega_{\widetilde{S}}\simeq \text{im}(\psi)$ since $\psi$ is injective. Next $\psi$ lifts to a map of complexes $\phi: \mathbb{\widetilde{F}}^{*}\rightarrow \mathbb{\widetilde{F}}$ which gives us a resolution of a Gorenstein ring $\widetilde{S}/\omega_{\widetilde{S}}$ of codimension 4. 
In this case, we say that the resolution of $\widetilde{S}/\omega_{\widetilde{S}}$ is obtained by generic doubling of $\mathbb{F}$.



\section{Spinor structures on resolutions of Gorenstein ideals of codimension four}\label{sec:spinor}
\begin{definition} 
		Assume notation from Remark \ref{rmk1}. We say that a given resolution of $R/I$ has a spinor structure if there exists a map $\tilde{\bs a}_{3}: R \rightarrow {V}(\omega_{n-1})\otimes R$ such that the following diagram commutes
		\[
		\xymatrixrowsep{1.5pc} \xymatrixcolsep{6.0pc}
		\xymatrix{
			R\ar@{->}[dr]_{\bs a_3} \ar@{->}[r]^{S_2(\tilde{\bs a}_3)}& S_2(V(w_{n-1}))\otimes R  \ar@{->}[d]^{\p\otimes R} \\
			& \bigwedge\limits^{n-1} F_2}
		\]
		where $\bs a_3$  is the map given by the First Structure Theorem of Buchsbaum and Eisenbud in Theorem \ref{th1} and $\p$ is the map
		described in Section \ref{sec:p}.
\end{definition} 

Now we are ready to show the existence of a spinor structure on a length four minimal resolution of a Gorenstein ideal over a complete regular local rings and a polynomial ring over a field.  

\begin{theorem}\label{mainthm} $\phantom{eee}$
\begin{enumerate}
	\item	Let $(R,\mathfrak{m},k)$ be a Gorenstein local ring in which $2$ is a unit. Let $I\subset R$ be Gorenstein ideal of grade 4 and let $\mathbb F$ be a minimal free resolution of $R/I$ of the form
	\[\mathbb{F}:0\rightarrow{} F_4\xrightarrow{{\bs d}_4} F_3\xrightarrow{{\bs d}_3}F_2\xrightarrow{{\bs d}_2}F_1\xrightarrow{{\bs d}_1 }R\rightarrow 0.\] 
	Assume that the multiplication $Q:F_2\otimes F_2 \rightarrow F_4$ defined in Remark~\ref{rmk1}.(e)  is in hyperbolic form. Then there exists a spinor structure on $\mathbb{F}$. 
		\item The same conclusion holds when $R$ is a polynomial ring over an algebraically closed field $K$ of characteristic different from $2$ 		and $I$ is a homogeneous Gorenstein ideal of codimension $4$.
		\end{enumerate}
\end{theorem}

Before we prove Theorem \ref{mainthm}, let us explain its meaning more precisely.

\begin{remark} \label{rmkmainthm} $\phantom{dd}$ 
		\begin{enumerate}
	\item Assume the hypothesis of Theorem \ref{mainthm} and let $\mu(I)=n$.
The Buchsbaum-Eisenbud multiplier $\bs a_{3,\mathcal{K}}$ is the square of the spinor coordinate $\widetilde{\bs a}_{3,\mathcal{J}}$ for $\mathcal{K}=-\mathcal{J}\cup \mathcal{J}^c$, where $\mathcal{J}\subset [1, n-1]$ with $\ell(\mathcal{J})$ even, i.e when the multiindex $\mathcal{K}\subset \lbrace \pm 1,\ldots ,\pm (n-1)\rbrace$ of cardinality $n-1$ contains all the numbers $1,2,\ldots ,n-1$ (with arbitrary signs). 
The Buchsbaum-Eisenbud multipliers for $\mathcal{K}=-\mathcal{J}\cup \mathcal{J}^c$, where $\mathcal{J}\subset [1, n-1]$ with $\ell(\mathcal{J})$ odd are all zero.
Note that we already made a choice that im$(d_3)$ is in the connected component Spin$(2(n-1))/P_{n-2}$. If we make another choice, then the opposite happens:
the multiplier $\bs a_{3,\mathcal{K}}$ is the square of the spinor coordinate $\widetilde{\bs a}_{3,\mathcal{J}}$ for $\mathcal{K}=-\mathcal{J}\cup \mathcal{J}^c$, where $\mathcal{J}\subset [1, n-1]$ with $\ell(\mathcal{J})$ odd, i.e when the multi-index $\mathcal{K}\subset \lbrace \pm 1,\ldots ,\pm (n-1)\rbrace$ of cardinality $n-1$ contains all the numbers $1,2,\ldots ,n-1$ (with arbitrary signs). 
The Buchsbaum-Eisenbud multipliers for $\mathcal{K}=-\mathcal{J}\cup \mathcal{J}^c$, where $\mathcal{J}\subset [1, n-1]$ with $\ell(\mathcal{J})$ even are all zero.

\item For other indices $\mathcal{K}$ (i.e. those where some numbers $\pm i$ are missing for some $1\le i\le n-1$), the multiplier $\bs a_{3,\mathcal{K}}$ is given by the expression from Lemma \ref{lem:mapp}.
	\end{enumerate} 
	\end{remark}

We need the following lemma for the proof of Theorem \ref{mainthm}.

\begin{lemma}\label{lemma}
Let $R$ be a polynomial ring over a field $K$ of characteristic different than 2 and let $I$ be a codimension four homogeneous Gorenstein ideal of $R$ with $\mu(I)=n$. Let 
\[\mathbb F: 0\rightarrow R(-d)\xrightarrow{d_4}F\xrightarrow{d_3}G\simeq G^*\xrightarrow{d_2}F^*\xrightarrow{d_1}R\] be a graded free resolution of $R/I$. Then $G$ has a hyperbolic basis.
\end{lemma}
\begin{proof}
Let $$G=(\bigoplus\limits_{i=1}^{n-1}R(-a_i))\bigoplus (\bigoplus\limits_{i=1}^{n-1}R(-a_i+d))$$ where $a_1\leq \cdots\leq a_{n-1}\leq d-a_{n-1}\leq \cdots\leq d-a_1$. Let ${\tilde Q}:G\rightarrow G^*$ be a symmetric isomorphism that induces a bilinear map $Q: G\otimes G \rightarrow R$. Now choose a basis element $e_1$ of highest degree $d-a_1$ of $G$. Then there exists a complementary $e_1'$ of lowest degree $a_1$ such that $Q(e_1, e_1')=1$ since $Q$ is non-degenerate; and, for any basis element $e_j$ in $G$, we have $Q(e_1, e_j)$ is a constant. 
Let $W_1=Re_1+Re_1'$, and let $W_1^{\perp}$ denote the orthogonal complement of $W_1$ in $G$.

If $d-a_1=a_1$, then the entries of the matrix with respect to the map ${\tilde Q}$ belong to $K$. Then, by the change of basis, one can transform ${\tilde Q}$ into a hyperbolic form. In the case that $d-a_1\neq a_1$, there exists a $2\times 2$ submatrix of ${\tilde Q}$ with respect to the basis $\{e_1, e'_1\}$ of the form   
$${\tilde Q}_1=\begin{bmatrix}
0 & 1 \\
1 & z
\end{bmatrix}, \;\; \text{where} \; z\in R \; \text{is of positive degree.}$$
 Next choose a $2\times 2$ matrix 
 $A=\begin{bmatrix} 1  & -\frac{1}{2}z \\ 
0  & 1 \end{bmatrix}$ such that $A^T{\tilde Q}_1A=\begin{bmatrix} 0 & 1\\ 1 & 0\end{bmatrix}$. Then $A^T{\tilde Q}_1A$ restricted to $W_1$ is in the hyperbolic form.

 If the basis elements of $W_1^{\perp}$ are of the same degree, then we are done. 
Otherwise, repeating the argument above for ${\tilde Q}_1$, we can always construct an hyperbolic pair using the highest and the lowest degree basis elements of $W_1^{\perp}$.

Continuing in this way, we get $W_k=W_{k-1}\oplus Re_k+Re_k'$ for some $k \leq n-1$ such that 
the lowest and highest degree basis elements of $W_k$ can be transformed into a hyperbolic pair and the basis elements of $W_k^{\perp}$ are of the same degree. 
Then the entries of the matrix with respect to the map $\varphi$ restricted to $W_k^{\perp}$ belong to $K$, and, by the change of basis, one can transform this matrix into a hyperbolic form. This finishes the proof of the lemma.\end{proof}

\begin{remark} Let $R$ be a complete regular local ring (respectively a polynomial ring over a field) in which 2 is a unit and let $I\subset R$ be a Gorenstein ideal (respectively  a graded Gorenstein ideal) of codimension four with $\mu(I)=n$. Then, by Remark ~\ref{rmk1}.(d) (respectively Lemma ~\ref{lemma}), a minimial free resolution of $R/I$ is of the form 
			 \[\mathbb{F}: 0\rightarrow{}R\xrightarrow{\bs d_1^t} F_1^* \xrightarrow{\bs {\tilde Q} \bs d_2^t} F_2\xrightarrow{\bs d_2} F_1\xrightarrow{\bs d_1} R.\]

\end{remark}
\begin{proof}[Proof of Theorem~\ref{mainthm}]  We first prove part (1). Let $\mu(I)=n$.
	Let $\{g_1,\ldots,g_n\}$ be a basis of $F_1$. By Remark \ref{rmk1}, 
	there exists a multiplicative structure on the resolution $\mathbb{F}$ for which the multiplication $F_2\otimes F_2\rightarrow F_4$ 
	can be brought to standard hyperbolic form. Let $\lbrace e_1,\ldots ,e_{n-1}, e_{-n+1},\ldots ,e_{-1}\rbrace$ be a hyperbolic basis of $F_2$.
	Such basis exists over $R$ by Remark \ref{rmk1}.
	By Leibniz formula, we see that the image of $d_3$ is an isotropic submodule.
	
	Let $R_{(0)}$ be the field of fractions of $R$ .
	The complex $\mathbb{F}\otimes_R R_{(0)}$ is split exact. We can choose a hyperbolic basis $\lbrace e'_1,\ldots , e'_{n-1}, e'_{-n+1}, \ldots, e'_{-1}\rbrace$ of $F_2\otimes_R R_{(0)}$ such that im$(\bs d_3\otimes_R R_{(0)})$ is the span of $e'_1,\ldots ,e'_{n-1}$. This subspace occurs then in  the connected component of the isotropic Grassmannian IGrass$(n-1, F_2\otimes_R R_{(0)})$
	corresponding to Spin$(2n-2)/P_{n-2}$.  The subspace im$(d_3\otimes_R R_{(0)})$ is isotropic, and  it has spinor coordinates in $V(\omega_{n-2})$. 
	
	The Pl\"ucker coordinates of this subspace are Buchsbaum-Eisenbud multipliers, i.e., the coordinates of the map $a_3$. The Buchsbaum-Eisenbud multipliers
	must therefore have expressions  given by the Lemma \ref{lem:mapp} in terms of spinor coordinates. 
	By equivariance such relations have to be satisfied for every choice of hyperbolic basis of $F_2\otimes_R R_{(0)}$, and for every choice of basis in $F_1\otimes R_{(0)}$.
	
	The only remaining thing is to check that in our original bases $\lbrace g_1,\ldots ,g_n\rbrace$ and \break $\lbrace e_1,\ldots ,e_{n-1}, e_{-1},\ldots ,e_{-n+1}\rbrace$ the spinor coordinates are not just in $R_{(0)}$ but in $R$. But for each of them its square is the appropriate Buchsbaum-Eisenbud multiplier which is in $R$. Since $R$ is normal, we conclude that all spinor coordinates are in $R$. This means all the relations from Lemma \ref{lem:mapp} are satisfied in $R$ since they hold in $R_{(0)}$. This proves part (1). 
	
	In case of polynomial rings over fields, $F_2$ has a hyperbolic basis by Lemma~\ref{lemma}. Thus, the proof of part (2) of the theorem follows the same as in part (1).	
 \end{proof}
 
As a consequence of Remark~\ref{rmk1}.(d) and Theorem~\ref{mainthm}, we get the following corollary.
  
 \begin{corollary} If $R$ is a complete regular local ring in which $2$ is a unit and $I$ is a Gorenstein ideal of grade 4, then there exists a spinor structure on a minimal free resolution of $R/I$.
 \end{corollary} 

\begin{remark} \label{rmk3}
$\phantom{dd}$ 
		\begin{enumerate}
		\item The hypotheses of $R$ being a domain could probably be dropped. It  would follow by localizing at the set of non-zero divisors in $R$, but probably requires some additional work on half-spinor representations over commutative rings.
		\item The normality assumption also might not be necessary.
		\item As the examples in the next section show, often the spinor structures exist even without characteristic $\ne 2$ assumptions. However this involves finding a hyperbolic basis for $F_2$ which was established by Kustin and Miller only under this assumption.

\item Spinor coordinates are in the radical of ideal $I$. If $I$ is a radical ideal, then spinor coordinates are in $I$.
We do not know any Gorenstein ideal $I$ of codimension $4$ for which the spinor coordinates are not in the ideal $I$.
	\end{enumerate} 
\end{remark}

\section{Examples of spinor structures on resolutions of codimension four Gorenstein ideals}\label{sec:examples}
In this section, we give explicit calculations of spinor structures on resolutions of well-known Gorenstein ideals with $4$, $6$, and $9$ generators. 
In some of these examples, we see that the spinor structure occurs even under weaker assumptions on $R$ than claimed in Theorem \ref{mainthm}.
The first two examples are also discussed in Reid's paper, see \cite{R}. 

\begin{example}\label{spinci}
		Let $R$ be an arbitrary commutative ring and  $\mathcal{K}(x_1, x_2, x_3, x_4; R)_{\bullet}$ be the Koszul complex resolving a complete intersection in codimension $4$ on elements $x_1, x_2, x_3, x_4$ from R. Let $F=R^4$. Then $\mathcal{K}(x_1,x_2,x_3,x_4; R)$ is a resolution of $R/I$ of the form 
	$$\mathcal{K}(x_1,x_2,x_3,x_4; R): 0\rightarrow{}\bigwedge\limits^{4}F\xrightarrow{{\bs d}_1^t}\bigwedge\limits^{3} F\xrightarrow{\bs {\tilde Q}{\bs d}_2^t}\bigwedge\limits^{2}F\xrightarrow{{\bs d}_2}\bigwedge\limits^{1}F\xrightarrow{{\bs d}_1}\bigwedge\limits^{0}F$$ 
	with ${\bs d}_1=[\begin{smallmatrix} x_1 & x_2 & x_3 & x_4\end{smallmatrix}]$, $
	{\bs d}_2=\left[\begin{smallmatrix}
	-x_4 & 0 &0 &0 & x_3 &  -x_2\\
	0 & -x_4 & 0 & -x_3 & 0 & x_1\\
	0 & 0 & -x_4 & x_2 & -x_1 & 0\\
	x_1 & x_2 & x_3 & 0 & 0 & 0
	\end{smallmatrix} \right]
	$. 
	Our calculation of the matrices $\bs d_1$ and $\bs d_2$ gives us the map ${\tilde Q}:\bigwedge\limits^2F\rightarrow \bigwedge\limits^2F^*$ in the hyperbolic form, that is, ${\tilde Q}=\left [\begin{smallmatrix}
	0 & {\rm I}_3\\
	{\rm I}_3 & 0
	\end{smallmatrix} \right ]$. Note that the quadratic form $Q: \bigwedge \limits^2 F \otimes_R \bigwedge \limits^2 F \rightarrow \bigwedge \limits^4 F$ is just the exterior multiplication.

We have $\bigwedge\limits^3(\bigwedge\limits^2 F)=S_{2,2,2,0}(F)\oplus S_{3,1,1,1}(F)$; see \cite[Proposition 2.3.9]{JW}, and  $R={\rm{Sym}}(F)$. Note that $R=\bigwedge\limits^4F\otimes \bigwedge\limits^4F$ and there is a map $(\bs m_{13} \otimes 1 \otimes 1)\circ (\Delta\otimes \Delta): R\rightarrow S_2F\otimes \bigwedge\limits^3 F \otimes \bigwedge\limits^3 F$ where $\bs m_{13}: F\otimes F\rightarrow S_2F$ is a multiplication map and $\Delta: \bigwedge\limits^i F\rightarrow  F\otimes \bigwedge\limits^{i-1} F $ is the diagonal map (for details; see \cite[Subsection 1.1.1]{JW}). 
Using $\Delta$, we get $\bs p_{24}: \bigwedge\limits^3 F \otimes \bigwedge\limits^3 F \twoheadrightarrow  \bigwedge\limits^3(\bigwedge\limits^2 F)$  by the commutativity of the following diagram 
\[
		\xymatrixrowsep{1.5pc} \xymatrixcolsep{6.0pc}
		\xymatrix{
			\bigwedge\limits^3 F \otimes \bigwedge\limits^3 F \ar@{->}[d] \ar@{->>}[r]^{\bs p_{24}}& \bigwedge\limits^3(\bigwedge\limits^2 F)\\
		\bigwedge\limits^2 F \otimes \bigwedge\limits^2 F \otimes \bigwedge\limits^2 F \ar@{->>}[ur]}\]
Thus there is a map $(1\otimes \bs p_{24}) \circ \sigma: R\rightarrow S_2F \otimes \bigwedge\limits^3(\bigwedge\limits^2 F)$, where $\sigma=(\bs m_{13} \otimes 1 \otimes 1)\circ (\Delta\otimes \Delta)$. Now we set $\bs a_3=(1\otimes \bs p_{24}) \circ \sigma$. It is clear that $\bs a_3$ goes to the summand $S_{2,2,2,0}(F)$ using the below diagram:
		\[
		\xymatrixrowsep{1.2pc} \xymatrixcolsep{2.2pc}
	\xymatrix{
		\bigwedge\limits^3 F \otimes \bigwedge\limits^3 F \ar@{->>}[dr]^{\bs p_{24}} \ar@{->}[r]\ar@{->}[d]  & \bigwedge\limits^2 F \otimes \bigwedge\limits^2 F \otimes \bigwedge\limits^2 F \ar@{->}[d]\\
		S_{2,2,2,0} F  \ar@{->}[r] & \bigwedge\limits^3(\bigwedge\limits^2 F) \\
	}
\]
In fact, it is the second symmetric power of the map
	$$\widetilde{\bs a}_3:R\rightarrow \bigwedge \limits^3 F$$ 
	sending $1$ to $x_1e_2\wedge e_3\wedge e_4-x_2 e_1\wedge e_3\wedge e_4+x_3e_1\wedge e_2\wedge e_4-x_4e_1 \wedge e_2\wedge e_3$. This
	last map gives us the spinor structure.



	
	Let us interpret this in terms of the root systems. Here we deal with a root system $D_3$ which is just $A_3$.  So the vector representation ${G}$ of rank $6$ can be considered as
	the second fundamental representation $\bigwedge\limits^2 {H}$ where ${H}$ is the $4$-dimensional space.
	Finding the structure map $\bs a_3$, we see that it is given by $R\rightarrow \bigwedge\limits^3 (\bigwedge\limits^2{H})$. The map $\widetilde {\bs a}_3$ is just the map from $R$ to ${H}$ and it allows us to identify ${H}$ and $F$. 
\end{example}


Next we look at  a hyperplane section of a codimension three Gorenstein ideal of Pfaffians of a skew-symmetric matrix.

\begin{example}\label{spinpfaff}  Let $R$ be a polynomial ring over ${\mathbb{Z}}$ on the entries of $(2n+1) \times (2n+1)$ skew-symmetric matrix $X=(x_{ij})$ with $x_{ij}=-x_{ji}$, and an additional variable $y$. Consider the resolution given by 
		\begin{equation}\label{ex:hyperplane}
		\mathbb F: (0\rightarrow R\xrightarrow{y}R\rightarrow 0)\otimes (0\rightarrow R\xrightarrow{\bs \partial_1^*} R^{2n+1}\xrightarrow{\bs \partial_2}R^{2n+1}\xrightarrow{\bs \partial_1}R\rightarrow 0)
		\end{equation}
		where $\bs \partial_1=\begin{bmatrix} (-1)^i {\rm Pf}([1,2n+1]\setminus \{ i \}, X) \end{bmatrix}$ where $[1, 2n + 1]=\{1, \ldots, 2n+1\}$ and $\bs \partial_2=X$. Here ${\rm Pf}(\mathcal{L},{X})$ denotes the Pfaffian of the submatrix of ${X}$ on rows and columns from $\mathcal{L}$ where $\mathcal L \subset [1, 2n + 1]$. Note that the resolution $\mathbb F$ in (\ref{ex:hyperplane}) is a hyperplane section of a codimension $3$ Gorenstein ideal of Pfaffians of a skew-symmetric matrix ${X}$.
		
	The matrix of the differential $\bs d_2:R^{2n+1} \oplus
	R^{2n+1} \rightarrow R \oplus R^{2n+1}$ of $\mathbb F$ is given by 
	$$
	{\bs d_2}=\begin{bmatrix}
	\bs \partial_1& 0\\
	-y{\rm I}_{2n+1} & {X}
	\end{bmatrix}.
	$$
	Our calculation directly gives $${\tilde Q}:R^{2n+1} \oplus R^{2n+1}\rightarrow (R^{2n+1} \oplus R^{2n+1})^{*}$$ which is in the hyperbolic form up to permutation, that is, ${\tilde Q}=\left [\begin{smallmatrix}
	0 & {\rm I}_{2n+1}\\
	{\rm I}_{2n+1} & 0\\
	\end{smallmatrix} \right ]$.
		
		\smallskip
	 	Denote the $i$th column of $\bs d_2$ by $e_i$. Set $e_{-i}=e_{2n+1+i}$. Hence the associated hyperbolic basis of the middle module $R^{2n+1} \oplus R^{2n+1}$ of $\mathbb F$ is $\{e_1, \ldots, e_{2n+1}, e_{-1}, \ldots, e_{-(2n+1)}\}$.
		
		
Let $\mathcal{I}$ and $\mathcal N$ be index sets of $[1, 2n + 1]=\{1, \ldots, 2n+1\}$ and $-\mathcal{I} \cup ([1, 2n + 1] \setminus \mathcal{I})$, respectively. Let $\mathcal J$ be an index set of $[1, 2n + 2]=\{1, \ldots, 2n+2\}$ where the cardinality of $\mathcal J$ is equal to $2n+1$, 
and let $(\bs d_2)_{\mathcal{J,} \mathcal{N}}$ denote the $(2n + 1)\times (2n + 1)$ minors $\bs d_2$ on rows $\mathcal J$ and columns $\mathcal N$. Note that $(2n + 1)\times (2n + 1)$ minors of the matrix $\bs d_2$ corresponding to $\mathcal{J}$ rows and $\mathcal N$ columns is of the form: 
$$(\bs d_2)_{\mathcal{J}, \mathcal{N}}=
\begin{cases} 
y^{2n+1-\ell(\mathcal {I})}({\rm Pf}(\mathcal I, X))^2, \;\; \text{if}  \; \mathcal J=[1,2n+2]\setminus \{1\}, \\
{\rm Pf}([1,2n+1]\setminus \{ i \}, X)y^{2n-\ell(\mathcal {I})}({\rm Pf}(\mathcal I, X))^2, \;\; \text{if}  \; \mathcal J=[1,2n+2]\setminus \{i\}.
\end{cases}
$$
We see that the coordinate of $\bs a_3$ corresponding to $\mathcal N$ is equal to $y^{2n-\ell(\mathcal {I})}({\rm Pf}(\mathcal I, X))^2$. This means that there exists a map $\widetilde{\bs a}_3$ from $R$ to a half-spinor representation
	${V}(\omega_{2n}, D_{2n+1})$ sending 1 to the combination of basis vectors ${v}_\mathcal{L}$ of weights
	$(\pm \frac{1}{2},\cdots,\pm \frac{1}{2})$ with $2n+1$ coordinates and even number of minuses (indicated by multi-index $\mathcal{L}$) with coefficients $$\pm y^{n-\ell(\mathcal{L})/2} {\rm Pf}(\mathcal{L},{X}).$$ 
	It is easy to check that the map $\widetilde{\bs a}_{3}$ gives the spinor structure on our resolution. If cardinality of $\mathcal{L}$ is $2n$, then ${\rm Pf}(\mathcal{L},{X})$ is the spinor coordinate. 
\end{example}

\begin{example}\label{tomjerry} In the 9-generator case, we have two examples of resolutions where none of the minimal generators are spinor coordinates. 
Note that we cannot explicitly get a hyperbolic basis over $\mathbb Q$, but, over $\mathbb C$, we can by Theorem \ref{mainthm}.(2).

\begin{enumerate}
\item The ring $R$ is a polynomial ring  in 9 variables on the entries of  $3\times 3$ generic matrix $X$ over a field $K$ of characteristic different than $2$. The ideal $I$ is generated by $2\times 2$-minors of the matrix $X$. 
\item The ring $R$ is a polynomial ring in 8 variables over a field $K$ of characteristic different from $2$ and $I$ is the ideal of the generated by the equation of the Segre embedding $\mathbf P^1\times \mathbf P^1\times \mathbf P^1$ into $\mathbf P^7$.
\end{enumerate}
	We observe that degree of all spinor coordinates is $3$ whereas minimal generators of $I$ are of degree $2$. Hence none of the minimal generators of $I$ are spinor coordinates. \end{example}

\section{Generic doubling of almost complete intersection and Kustin-Miller model }\label{sec:doublingsci}

In this section, we discuss generic doubling of almost complete intersection of codimension $3$ which leads to a specialization of Kustin-Miller model.
All polynomial rings $R$ are over a field $K$ of arbitrary characteristic. The spinor coordinates still exist since they are obtained by direct calculations.

\subsection{Kustin-Miller Model (KMM)} \label{KMmodel}

Next we recall the well-known Kustin-Miller family of ideals associated to a $3\times 4$ matrix, a $4$-vector and a variable. For details, see \cite{KM1,KM3}. \\
Let $R$ be a polynomial ring over $K$ with the indeterminates as the entries of $X$ and $M$ where
$${X}=\left[\begin{smallmatrix*} x_1 \\ x_2 \\ x_3 \\ x_4\end{smallmatrix*}\right]\;\; \text{and}\;\;  M=\begin{bmatrix*} a_{11} & a_{12} & a_{13} & a_{14} \\
		a_{21} & a_{22} & a_{23} & a_{24} \\
		a_{31} & a_{32} & a_{33} & a_{34} 
		\end{bmatrix*}.$$ 
		We set $q_i=\sum_{j=1}^4 a_{ij} x_j=a_{i1}x_1+a_{i2}x_2+a_{i3}x_3+a_{i4}x_4$ and
		\begin{align*}
		I&=\langle  q_1, q_2, q_3, x_1v+{ M}_{123; 234}, x_2v-{ M}_{123; 134}, x_3v+{M}_{123; 124}, x_4v-{M}_{123; 123}\rangle,
		\end{align*}
 where ${M}_{\mathcal{K}; \mathcal{L}}$ is the minor of the submatrix of ${M}$ involving $\mathcal{K}$ rows and $\mathcal{L}$ columns.  Let $\bs s$ be an $12\times 12$ exchange matrix with entries of the form $$\bs s_{ij}=\begin{cases} 1, & j=12-i+1\\ 
0, & j\neq 12-i+1.
\end{cases}$$
Note that $\bs s$ can be put in the form ${\bs {\tilde Q}}=\begin{bmatrix}
		0 & {\rm I}_{6}\\
		{\rm I}_6 & 0\\
		\end{bmatrix}
		$ up to permutation of columns.
Then minimal free resolution for $I$ is given by
		\[0\rightarrow R\xrightarrow{{\bs d}_1^t} R^7 \xrightarrow{\bs {s}{\bs D^t}}R^{12}\xrightarrow{{\bs D}}R^7 \xrightarrow{{\bs d}_1}R\rightarrow R/I\rightarrow 0\]
		where  
		$${\bs d}_1=\begin{bmatrix}
		q_1 & q_2 & q_3 & x_1v+{M}_{123; 234} & x_2v-{M}_{123; 134} & x_3v+{M}_{123; 124} & x_4v-{M}_{123; 123}\\
		\end{bmatrix}
		,$$  and the matrix ${\bs D}$ is
		{\small
		\setcounter{MaxMatrixCols}{12}
		$$\begin{bmatrix*} -q_2& -q_3&0& {M}_{23; 34} & {M}_{23; 24} & {M}_{23; 23} & {M}_{23; 14} & {M}_{23; 13} & {M}_{23; 12} & -v & 0 & 0\\ 
		q_1 & 0 &-q_3& -{M}_{13; 34} & -{M}_{13; 24} & -{M}_{13; 23} & -{M}_{13; 14} & -{M}_{13; 13} & -{M}_{13; 12} & 0 & -v & 0 \\ 0& q_1 & q_2& {M}_{12; 34} & {M}_{12; 24} & {M}_{12; 23} & {M}_{12; 14} & {M}_{12; 13} & {M}_{12; 12} & 0 & 0 & -v\\ 0& 0 & 0& -x_2 & x_3 & -x_4 & 0 & 0 & 0& a_{11} & a_{21} & a_{31}\\ 
		0& 0 & 0& x_1 & 0 & 0 & -x_3 & x_4 & 0& a_{12} & a_{22} & a_{32}\\
		0& 0 & 0& 0 & -x_1 & 0 & x_2& 0 & -x_4 & a_{13} & a_{23} & a_{33}\\
		0& 0 & 0& 0 & 0 & x_1 & 0 & -x_2 & x_3 & a_{14} & a_{24} & a_{34}
		\end{bmatrix*}.$$

\subsection{Spinor Coordinates of the Kustin-Miller model}
In this subsection, we calculate the spinor coordinates on the Kustin-Miller model.

\begin{proposition}\label{Kmspin} The spinor coordinates of KMM with 7 generators are given in Table \ref{tab:7gens}.
	
\end{proposition}
\begin{proof}
	We find the Buchsbaum-Eisenbud map $\bs a_3$. Let us denote the $i$th column of the matrix ${\bs D}$ in Section~\ref{KMmodel} by $e_i$ and $e_{-i}=e_{13-i}$. Then $\{e_1,\ldots,e_6,e_{-1},\ldots, e_{-6}\}$ is the associated hyperbolic basis of $R^{12}$. 
	
	Computing $6\times 6$ minors of ${\bs D}$,  we see that the coordinates of $\bs a_3$ corresponding to the multi-index $\mathcal{J}\cup \mathcal{J}^c$ where $\mathcal{J}\subset [1,6]$ are of odd cardinality.  Next we record nonzero spinor coordinates of $\widetilde{\bs a}_3$ as follows:
		\begin{table}[h]{\renewcommand{\arraystretch}{1.3}
			\caption{Spinor coordinates of KMM with 7 generators}
			\label{tab:7gens}
			\scalebox{1}{
			\begin{tabular}{|l|c|}
			\hline
				\multicolumn{1}{|c|}{Cases}& $\widetilde{\bs a}_{3,\mathcal{J}}$\\
			\hline
				$\mathcal{J}=\{i\}$ for $i=1,2,3$ &  $\pm x_1q_i$ \\
				\hline
				$\mathcal{J}=\{5,6\}\cup \{i\}$ for $i=1,2,3$ & $\pm x_2q_i$\\
				\hline
				$\mathcal{J}=\{4,6\}\cup \{i\}$ for $i=1,2,3$ & $\pm x_3q_i$ \\
				\hline 
				$\mathcal{J}=\{4,5\}\cup \{i\}$ for $i=1,2,3$ & $\pm x_4q_i$\\
				\hline 
				$\mathcal{J}=\{1,2,3\}$ & $(x_1v+{M}_{123; 234})$\\
				\hline
				$\mathcal{J}=\{1,2,3,5,6\}$ & $ (x_2v-{M}_{123; 134})$\\
				\hline 
				$\mathcal{J}=\{1,2,3,4,6\}$ & $(x_3v+{M}_{123; 124})$\\
				\hline
				$\mathcal{J}=\{1,2,3,4,5\}$ & $ (x_4v-{M}_{123; 124})$\\
				\hline
				$\mathcal{J}=\{k,\ell \}\cup \{r\}$, $k,\ell\in \{1,2,3\}$ $r\in\{4,5,6\}$ & $(a_{mj}q_n-a_{nj}q_m)$ $m,n\in \{1,2,3\}$, $j\in \{1,2,3,4\}$\\
				\hline
			\end{tabular}}}
		\end{table}
	
	This means that there exists a map $\widetilde{\bs a}_3$ from $R$ to half-spinor representation
	$V(\omega_{6}, D_{6})$ sending 1 to the combination of basis vectors $v_{\mathcal{L}}$ of weights
	$(\pm \frac{1}{2},\cdots,\pm \frac{1}{2})$ with coordinates $\widetilde{\bs a}_{3,\mathcal{L}}$ and odd number of minuses (indicated by multi-index $\mathcal{L}$). One can check that the map $\widetilde{\bs a}_{3}$ gives the spinor structure on the resolution in Section \ref{KMmodel}. From Table~\ref{tab:7gens}, observe that four minimal generators are spinor coordinates.  
	%
\end{proof}

\subsection{Generic doubling of an almost complete intersection}\label{genericdoublingaci}
Generic doubling of an almost complete intersection of codimension $3$ leads to a specialization of the KMM. After deformation of such a specialization one gets the KMM given in Section~\ref{KMmodel}.

Let $R=K[c_{ij},u_{kl}]$ be a polynomial ring over $K$ where the variables $c_{ij}$ are skew-symmetric   in $i,j$ and variables $u_{kl}$ are generic variables for $1\leq k,l\leq 3$. Consider a $3\times 3$ generic skew symmetric matrix $C=(c_{ij})$ and a generic matrix as
$$
{N}=\begin{bmatrix}
-u_{11} & u_{12} & -u_{13}\\
-u_{21} & u_{22} & -u_{23}\\
-u_{31} & u_{32} & -u_{33} 
\end{bmatrix}.$$
Let $J=\langle q_1,q_2,q_3,-N_{123;123}\rangle$ where $q_1=c_{23}u_{11}-c_{13}u_{12}+c_{12}u_{13},$
$q_2=c_{23}u_{21}-c_{13}u_{22}+c_{12}u_{23},$
$q_3=c_{23}u_{31}-c_{13}u_{32}+c_{12}u_{33},$ 
and where ${N}_{\mathcal{J};\mathcal{K}}$ is the submatrix of ${N}$ involving $\mathcal{J}$ rows and $\mathcal{K}$ columns.  By  \cite[Proposition 2.4 ]{CVW}, we get a minimal free resolution of $R/J$ as 
\begin{equation}\label{aci2}
\mathbb{F}: 0\rightarrow R^3\xrightarrow{{\bs d}_3} R^{6}\xrightarrow{{\bs d}_2} R^{4}\xrightarrow{{\bs d}_1} R\rightarrow R/J\rightarrow 0
\end{equation}
where 
\begin{align*}
{\bs d}_1&=\begin{bmatrix}
q_1 & q_2 & q_3 & -{N}_{123;123}
\end{bmatrix},\\[.5em]
{\bs d}_2&=\begin{bmatrix}
-q_2 & -q_3 & 0 & {N}_{23;12} & {N}_{23;13} & {N}_{23;23}\\
q_1 & 0 & -q_3 & -{N}_{13;12} & -{N}_{13;13} & -{N}_{13;23}\\
0 & q_1 & q_2 & {N}_{12;12} & {N}_{12;13} & {N}_{12;23}\\
0 & 0 & 0 & -c_{12} & c_{13} & -c_{23}\\
\end{bmatrix},\\[.5em]
{\bs d}_3&=\begin{bmatrix}
0& -c_{12} & c_{13}\\
c_{12} & 0 & -c_{23}\\
-c_{13} & c_{23} & 0\\
-u_{11} & u_{12} & -u_{13}\\
-u_{21} & u_{22}& -u_{23}\\
-u_{31}& u_{32} & -u_{33}.
\end{bmatrix}
\end{align*}
Next we study the generic doubling of the resolution in (\ref{aci2}). Applying ${\rm Hom}_R(-,R)$ to $\mathbb{F}$, one gets an acyclic complex
\[\mathbb{F}^{*}: 0\rightarrow R\xrightarrow{{\bs d}_3^{*}} R^{4}\xrightarrow{{\bs d}_2^{*}} R^{6}\xrightarrow{{\bs d}_1^{*}} R^3\rightarrow \omega_{R/J}\rightarrow 0\] 
where ${\bs d}_1^*=-{\bs d}_3^T$, ${\bs d}_2^*=-{\bs d}_2^T$ and ${\bs d}_3^*=-{\bs d}_1^T$. We compute  ${\rm Hom}_{R/J}(\omega_{R/J},R/J)$ by using Macaulay 2 \cite{GS},  which is generated by the image of the
following matrix: 
$$\mathcal{H}=\begin{bmatrix}
-c_{23} & {N}_{23;23} & {N}_{13;23} & {N}_{12;23}\\
-c_{13} &-{N}_{23;13}  & -{N}_{13;13} & -{N}_{12;13}\\
-c_{12} & {N}_{23;12} & {N}_{13;12} & {N}_{12;12}\\
\end{bmatrix}.
$$
Let $\widetilde{R}=R[\tau_1,\tau_2,\tau_3,\tau_4]$. We set 
\begin{align*}
s_1&=\tau_4c_{23}+\tau_1{N}_{23;23}+\tau_2{N}_{13;23}+\tau_3{N}_{12;23},\\[.1em]
s_2&=\tau_4c_{13}-\tau_1{ N}_{23;13}-\tau_2{N}_{13;13}-\tau_3{N}_{12;13},\\[.1em]
s_3&=\tau_4c_{12}+\tau_1{N}_{23;12}+\tau_2{N}_{13;12}+\tau_3{N}_{12;12}.
\end{align*}
Let ${M}^\prime=\begin{bmatrix}
-u_{11} & u_{12} & -u_{13}& \tau_1\\
-u_{21} & u_{22} & -u_{23}& -\tau_2\\
-u_{31} & u_{32} & -u_{33} & \tau_3
\end{bmatrix}$. Take $ {\bs \psi}_1=\begin{bmatrix}
s_1 & s_2 &s_3\\
\end{bmatrix}
$, and $${\bs\psi}_2=\begin{bmatrix}
	{M}^\prime_{23;14} & {M}^\prime_{23;24} & {M}^\prime_{23;34} & -\tau_4 & 0 & 0\\
	-{ M}^\prime_{13;14} & -{M}^\prime_{13;24} & -{M}^\prime_{13;34} & 0 & -\tau_4 & 0\\
	{M}^\prime_{12;14} & {M}^\prime_{12;24} & { M}^\prime_{12;34} & 0 & 0 &-\tau_4\\
	0 & 0 & 0 & -\tau_1 & \tau_2 & -\tau_3\\
	\end{bmatrix}, $$ where ${M}^\prime_{\mathcal{K};\mathcal{L}}$ denotes the minor of $M'$ involving $\mathcal{K}$ rows and $\mathcal{L}$ columns.

Set ${\bs \psi}_3=-{\bs \psi}_2^T$ and ${\bs \psi}_4=-{\bs \psi}_1^T$. Then $\overline {{\bs \psi}}_1:  \omega_{\widetilde{R}/J\widetilde{R}} \rightarrow \widetilde{R}/J\widetilde{R}$ lifts to  the  chain map $\bs \psi:\mathbb{\widetilde{F}}^*\rightarrow \mathbb{\widetilde{F}}$ of complexes as follows:
\[
\xymatrixrowsep{1.3pc} \xymatrixcolsep{2.2pc}
\xymatrix{
	\mathbb{\widetilde{F}}:0\ar@{->}[r]^{}&\widetilde{R}^3\ar@{->}[r]^{{\bs d}_3}&\widetilde{R}^6\ar@{->}[r]^{{\bs d}_2}& \widetilde{R}^4\ar@{->}[r]^{{\bs d}_1}&\widetilde{R}\ar@{->}[r]^{}&\ar@{->}[r]^{} \widetilde{R}/J\widetilde{R}&0\\
	\mathbb{\widetilde{F}}^*:0\ar@{->}[r]^{}&\widetilde{R}\ar@{->}[r]^{{\bs d}_3^*}\ar@{->}[u]^{{\bs \psi}_4}&\widetilde{R}^4\ar@{->}[r]^{{\bs d}_2^*}\ar@{->}[u]^{{\bs \psi}_3}& \widetilde{R}^6\ar@{->}[u]^{{\bs \psi}_2}\ar@{->}[r]^{{\bs d}_1^*}&\widetilde{R}^3\ar@{->}[r]^{}\ar@{->}[u]^{{\bs \psi}_1}&\omega_{\widetilde{R}/J\widetilde{R}}\ar@{->}[r]^{}\ar@{->}[u]^{\overline {{\bs \psi}}_1}&0\\
}
\]

Let $I=J\widetilde{R}+\langle s_1,s_2,s_3 \rangle$. Then the mapping cone with respect to ${\bs \psi}$ gives us a complex of the form
\begin{equation}\label{resKM}
\mathcal{C}({\bs \psi}):0\rightarrow \widetilde{R}\xrightarrow{{\bs \delta}_4} \widetilde{R}^7\xrightarrow{{\bs \delta}_3}\widetilde{R}^{12}\xrightarrow{{\bs \delta}_2} \widetilde{R}^7\xrightarrow{{\bs \delta}_1} \widetilde{R}\rightarrow \widetilde{R}/I\rightarrow 0
\end{equation}

where $${\bs \delta}_1=\begin{bmatrix}
{\bs d}_1 & {\bs \psi}_1
\end{bmatrix}, \;\; 
{\bs \delta}_2=\begin{bmatrix}
{\bs d}_2 &{\bs \psi}_2\\
{\bf 0} & -{\bs d}_3^T\\
\end{bmatrix}, \;\;
{\bs \delta}_3=\begin{bmatrix}
{\bs d}_3 & -{\bs \psi}_2^T\\
0 & -{\bs d}_2^T\\
\end{bmatrix},\;\;
{\bs \delta}_4=\begin{bmatrix}
-{\bs \psi}_1^T\\
-{\bs d}_1^T\\
\end{bmatrix}.$$

\begin{theorem}
		If we substitute $x_1=-c_{23}$, $x_2=-c_{13}$, $x_3=-c_{23}$, $x_4=0$ and sending ${M}$ to ${M}^\prime$ in Section~\ref{KMmodel}, then the resolution in (\ref{resKM})  is a specialization of KMM. 
\end{theorem}

In the next theorem, we give deformed ideal 
\begin{equation}\label{I(t)} I(t)=\langle q_1+t\tau_1 , q_2-\tau_2t, q_3+\tau_3t, -\det(N)-\tau_4t, s_1, s_2, s_3\rangle
\end{equation}
in the bigger polynomial ring $S=\widetilde{R}[t]$. Further, we show that the resolution of $I(t)$ is KMM. 

\begin{theorem}\label{KMgen}
		The ideal $I(t)$ above (in \ref{I(t)}) is Gorenstein  of codimension $4$ in $S$. The minimal resolution of $S/I(t)$ over $S$ is 
		\[0\rightarrow S\xrightarrow{{\bs \delta}_4(t)} S^7\xrightarrow{{\bs \delta}_3(t)}S^{12}\xrightarrow{{\bs \delta}_2(t)} S^7\xrightarrow{{\bs \delta}_1(t)} S\rightarrow S/I(t)\rightarrow 0.\] 
\end{theorem}
\begin{proof}
	Set $\lambda=(\tau_1,-\tau_2,\tau_3,-\tau_4,0,0,0)$ and $S=\widetilde{R}[t]$. Then  deformation of ideal $I$ along $\lambda$ is  $I(t)={\rm im}({\bs \delta}_1 (t))$ where 
	$${\bs \delta}_1(t)={\bs \delta}_1+ t\lambda,\; {\bs \delta}_2(t)=\begin{bmatrix}
	{\bs d}_2(t) &{\bs \psi}_2\\
	{\bs \phi}_2 & -{\bs d}_3^T\\
	\end{bmatrix},\; {\bs \delta}_3(t)=\begin{bmatrix}
	{\bs d}_3 & -{\bs \psi}_2^T\\
	-{\bs \phi}_2^T & -{\bs d}_2^T(t)\\
	\end{bmatrix}, $$
	$${\bs \delta}_4=\begin{bmatrix}
		-{\bs \psi}_1^T\\
		-{\bs d}_1^T(t)\\
	\end{bmatrix},\; \text{and} \;{\bs \phi}_2=\begin{bmatrix}0 & 0 & 0 &-t &0 &0\\
	0& 0& 0&0 & t & 0\\
	0& 0& 0&0 & 0 & -t \end{bmatrix}.$$
	
	By direct computation, we see that ${\rm im}([{\bs \delta}_1(t)]^T)=\ker([{\bs\delta}_2(t)]^T)$. Now we use Buchsbaum-Eisenbud exactness criteria given in \cite{BE3}. The rank condition is immediately satisfied. We claim that ${\rm depth}(I({\bs \delta}_i(t)))\geq 4$ for $1\leq i\leq 4$ where $I({\bs \delta}_i(t))$ denotes the ideal generated by $r_i\times r_i$ minors of $\delta_i(t)$. By construction we see ideals  $(I({\bs \delta}_i(t)),t)$ and $(I({\bs \delta}_i),t)$ are equal. So depth of $I({\bs \delta}_i(t), t)$ is at least  $5$. Thus the claim follows. 
\end{proof}

As an immediate we have the following result.

\begin{remark}
		If we substitute $x_1=-c_{23}$, $x_2=-c_{13}$, $x_3=-c_{23}$, $x_4=t$ and send entries of ${M}$ of Section \ref{KMmodel} to entries of ${M}^\prime$, then the resolution in Theorem \ref{KMgen} is a specialization of KMM. 
\end{remark}

\section{Generic doubling of the resolution of format (1,5,6,2)}\label{(1,5,6,2)D}

\subsection{Resolution of type $(1,5,6,2)$}\label{(1,5,6,2)}
We recall perfect ideals of codimension 3 with 5 generators of Cohen-Macaulay type 2. For details, see \cite{CLWK}.

Let $K$ be a field of characteristics different from two. Let $R$ be a polynomial ring over $K$ with variables $x_{i,j}$, $y_{i,j}$ $(1\leq i<j\leq 4)$, and $z_{i,j,k}$ $(1\leq i<j<k\leq 4)$. Note that $R$ is a bi-graded ring with $\deg(x_{i,j})=\deg(y_{i,j})=(1,0)$, and $\deg(z_{i,j,k})=(0,1)$.

We use $\Delta(ij,kl)$ to denote the $2\times 2$ minors of the matrix 
$$\begin{bmatrix*}
x_{1,2} & x_{1,3} & x_{1,4} & x_{2,3} & x_{2,4} & x_{3,4}\\
y_{1,2} & y_{1,3} & y_{1,4} & y_{2,3} & y_{2,4} & y_{3,4}
\end{bmatrix*} $$
corresponding to the columns labeled by $(i,j)$ and $(k,l)$. The cubic generators of bidegree $(2,1)$ are $u_{1,2,3}$, $u_{1,2,4}$, $u_{1,3,4}$ and $u_{2,3,4}$, where
{\tiny
	\begin{align*}
	u_{1,2,3}&=-2z_{2,3,4}\Delta(12,13)+2z_{1,3,4}\Delta(12,23)-2z_{1,2,4} \Delta(13,23)+z_{1,2,3}(\Delta(13,24)-\Delta(12,34)+\Delta(14,23))\\
	u_{1,2,4}&=2z_{2,3,4}\Delta(12,14)-2z_{1,3,4}\Delta(12,24)+z_{1,2,4}(\Delta(12,34)+\Delta(13,24)+\Delta(14,23))- 2z_{1,2,3}\Delta(14,24)\\
	u_{1,3,4}&=2z_{2,3,4}\Delta(13,14)+z_{1,3,4}(-\Delta(12,34)-\Delta(13,24)+\Delta(14,23))+2z_{1,2,4}\Delta(13,34)- 2z_{1,2,3}\Delta(14,34)\\
	u_{2,3,4}&=z_{2,3,4}(-\Delta(12,34)-\Delta(13,24)-\Delta(14,23))-2z_{1,3,4}\Delta(23,24)+2z_{1,2,4}\Delta(23,34)- 2z_{1,2,3}\Delta(24,34).
	\end{align*}}
The generator of degree $(4,0)$  is $u=b^2-4ac$, where 
\begin{align*}
a &=x_{1,2}x_{3,4}-x_{1,3}x_{2,4}+x_{1,4}x_{2,3},\\
b &=x_{1,2}y_{3,4}-x_{1,3}y_{2,4}+x_{1,4}y_{2,3}+x_{3,4}y_{1,2}-x_{2,4}y_{1,3}+x_{2,3}y_{1,4},\\
c &=y_{1,2}y_{3,4}-y_{1,3}y_{2,4}+y_{1,4}y_{2,3}.
\end{align*}

Let $u_j=\sum_{i\neq j}(-1)^i x_{i,j} z_{\hat{i}}$, $v_j=\sum_{i\neq j}(-1)^iy_{i,j}z_{\hat{i}}$, $\delta_1=\Delta(12,34)$, $\delta_2=\Delta(13,24)$, and $\delta_3=\Delta(14,23)$. Let  $J=\langle u_{2,3,4},u_{1,3,4},u_{1,2,4}, u_{1,2,3},u\rangle.$
%
%

We recall from \cite[Section 3]{CLWK} the deformed ideal $J(t)$ which is an ideal in the bigger polynomial ring $S=R[t]$. The deformed matrices ${\bs d}_2$ and ${\bs d}_3$ become

{\small
	\begin{align*}
	{\bs d}_2(t)&=\setlength\arraycolsep{2pt} \begin{bmatrix*}
	v_1 & u_1 & -\delta_1+\delta_2-\delta_3+t & 2\Delta(13,14) & -2\Delta(12,14) & -2\Delta(12,13) \\
	-v_2 & -u_2 &-2\Delta(23,24) & -\delta_1-\delta_2+\delta_3+t & 2\Delta(12,24) & 2 \Delta(12,23) \\
	v_3 & u_3 & 2\Delta(23,34) & 2\Delta(13,34) & -\delta_1-\delta_2-\delta_3-t & -2\Delta(13,23)\\
	v_4 & u_4 &2\Delta(24,34) & 2\Delta(14,34) & -2\Delta(14,24) &\delta_1-\delta_2-\delta_3+t\\
	0 & 0&-z_{2,3,4} & -z_{1,3,4} & z_{1,2,4} & z_{1,2,3}
	\end{bmatrix*}\\[.5em]
	{\bs d}_3(t)&=\begin{bmatrix*}
	b+t & 2a\\
	-2c & -b+t\\
	-v_1& -u_1\\
	v_2 & u_2\\
	v_3 & u_3\\
	-v_4 & -u_4
	\end{bmatrix*}.
	\end{align*}}
Set $u_{1,2,3}(t)=-u_{1,2,3}+z_{1,2,3}t$, $u_{1,2,4}(t)=-u_{1,2,4}+z_{1,2,4}t$, $u_{1,3,4}(t)=-u_{1,3,4}+z_{1,3,4}t$, $u_{2,3,4}(t)=-u_{2,3,4}+z_{2,3,4}t$, $u(t)=u-t^2$. By \cite[Section 3]{CLWK}, $$J(t)=\langle u_{2,3,4}(t),u_{1,3,4}(t),u_{1,2,4}(t), u_{1,2,3}(t),u(t)\rangle$$ is a perfect ideal of codimension three in $S$. The minimal free resolution of $S/J(t)$ over $S$ is
\begin{equation}\label{defP}
\mathbb{F}:0\xrightarrow{}S(-7)^2\xrightarrow{{\bs d}_3(t)} S(-5)^6\xrightarrow{{\bs d}_2(t)}S(-4)\oplus S(-3)^4 \xrightarrow{{\bs d}_1(t)}S.
\end{equation}

\subsection{Generic doubling of complex of type (1,5,6,2)} \label{newfamily}
We discuss  generic doubling of perfect ideals of codimension 3 with 5 generators of Cohen-Macaulay type 2 given in (\ref{defP}) above. Applying ${\rm Hom}_S(-,S)$ to $\mathbb{F}$ one gets an acyclic complex
\[\mathbb{F}^{*}: 0\rightarrow S\xrightarrow{{\bs d}_3^{*}} S^{5}\xrightarrow{{\bs d}_2^{*}} S^{6}\xrightarrow{{\bs d}_1^{*}} S^2\rightarrow \omega_{S/J(t)}\rightarrow 0\] 
where ${\bs d}_1(t)^*=-{\bs d}_3(t)^T$, ${\bs d}_2(t)^*=-{\bs d}_2(t)^T$ and ${\bs d}_3(t)^*=-{\bs d}_1(t)^T$. Then we compute  ${\rm Hom}_{S/J(t)}(\omega_{S/J(t)},S/J(t))$ by Macaulay 2 \cite{GS}, 
which is generated by the image of the matrix

$$\begin{bmatrix}
u_4 & u_3 & u_2 & u_1 & b-t & 2a\\
-v_4 & -v_3 & -v_2 & -v_1 & -2c & -b-t\\
\end{bmatrix}$$

Consider the bigger polynomial ring $\widetilde S=S[\tau_1, \tau_2, \tau_3, \tau_4, \tau_5, \tau_6]$, and 
$$f_1(t)=-\tau_1u_4-\tau_2u_3-\tau_3u_2-\tau_4u_1+\tau_5b+2a \tau_6-\tau_5t$$
$$f_2(t)=\tau_1v_4+\tau_2v_3+\tau_3v_2+\tau_4v_1-2c\tau_5-b\tau_6-\tau_6t.$$

Let ${\bs \psi}_1(t)=\begin{bmatrix}
f_1(t) & f_2(t)\\
\end{bmatrix}$ and ${\bs \psi}_2(t)$ is the transpose of the matrix given in Figure~\ref{mat:psi_2(t)} at the end of the paper. Then $\overline {{\bs \psi}}_1(t):  \omega_{\widetilde{S}/J(t)\widetilde{S}} \rightarrow \widetilde{S}/J(t)\widetilde{S}$ lifts to  the  chain map $\bs \psi:\mathbb{\widetilde{F}}^*\rightarrow \mathbb{\widetilde{F}}$ such that 

{\small
	\[
	\xymatrixrowsep{1.8pc} \xymatrixcolsep{2.2pc}
	\xymatrix{
		\mathbb{\widetilde{F}}:0\ar@{->}[r]^{}&\widetilde{S}^2\ar@{->}[r]^{{\bs d}_3(t)}&\widetilde{S}^6\ar@{->}[r]^{{\bs d}_2(t)}& \widetilde{S}^5\ar@{->}[r]^{{\bs d}_1(t)}&\widetilde{S}\ar@{->}[r]^{}&\ar@{->}[r]^{}\widetilde{S}/J(t)\widetilde{S}&0\\
		\mathbb{\widetilde{F}}^*:0\ar@{->}[r]^{}&\widetilde{S}\ar@{->}[r]^{{\bs d}_3(t)^*}\ar@{->}[u]^{-{\bs \psi}_1(t)^T}&\widetilde{S}^5\ar@{->}[r]^{{\bs d}_2(t)^*}\ar@{->}[u]^{-{\bs \psi}_2(t)^T}& \widetilde{S}^6\ar@{->}[u]^{{\bs \psi}_2(t)}\ar@{->}[r]^{{\bs d}_1(t)^*}&\widetilde{S}^2\ar@{->}[r]^{}\ar@{->}[u]^{{\bs \psi}_1(t)}&\omega_{\widetilde{S}/J(t)\widetilde{S}}\ar@{->}[r]^{}\ar@{->}[u]^{{\overline {{\bs \psi}}_1}(t)}&0\\
	}
	\]
}
Let $I(t)=J(t)\widetilde{S}+\langle f_1(t),f_2(t) \rangle$. Then the mapping cone with respect to ${\bs \psi}$ gives complex of the form
\begin{equation}\label{mapcone}
\mathcal{C}({\bs \psi}):0\rightarrow \widetilde{S}\xrightarrow{{\bs {\bs \delta}}_4(t)} \widetilde{S}^7\xrightarrow{{\bs {\bs \delta}}_3(t)}\widetilde{S}^{12}\xrightarrow{{\bs \delta}_2(t)} \widetilde{S}^7\xrightarrow{{\bs \delta}_1(t)} \widetilde{S}\rightarrow \widetilde{S}/I(t)\rightarrow 0
\end{equation}
with differentials 
$${\bs \delta}_1(t)=\begin{bmatrix}
{\bs d}_1(t) & {\bs \psi}_1(t)
\end{bmatrix}, \;\; {\bs \delta}_2(t)=\begin{bmatrix}
{\bs d}_2(t) &{\bs \psi}_2(t)\\
{\bf 0} & -{\bs d}_3(t)^T\\
\end{bmatrix},$$
$${\bs \delta}_3(t)=\begin{bmatrix}
{\bs d}_3(t) & -{\bs \psi}_2(t)^T\\
0 & -{\bs d}_2(t)^T\\
\end{bmatrix},\;\;
{\bs \delta}_4(t)=\begin{bmatrix}
-{\bs \psi}_1(t)^T\\
-{\bs d}_1(t)^T\\
\end{bmatrix}.$$
Note that $\bs \delta_3=\bs s \bs \delta_2^T$ where $\bs s$ is a $12\times 12$ exchange matrix with entries given by $$\bs s_{ij}=\begin{cases} 1, & j=12-i+1\\ 
0, & j\neq 12-i+1,
\end{cases}$$
and $\bs s$ can be put in the form ${\bs {\tilde Q}}=\begin{bmatrix}
		0 & {\rm I}_{6}\\
		{\rm I}_6 & 0\\
		\end{bmatrix}
		$ up to permutation of columns.
%

\begin{theorem}\label{thspin5}  
		The spinor coordinates of resolution (\ref{mapcone}) are given in Table \ref{tab:resn11}.
\end{theorem}
\begin{proof}
	By column operations in resolution (\ref{mapcone}), one gets differentials as
	$${\bs \delta}_1(t)=\begin{bmatrix}
	{\bs d}_1(t) & {\bs \psi}_1(t)
	\end{bmatrix},\;\; {\bs \delta}_2(t)=\begin{bmatrix}
	{\bs d}_2(t) &{\bs \psi}_2(t)\\
	{\bf 0} & -{\bs d}_3(t)^T\\
	\end{bmatrix},$$
	
	$${\bs \delta}_3(t)=\begin{bmatrix}
	-{\bs \psi}_2(t)^T &{\bs d}_3(t) \\
	-{\bs d}_2(t)^T &0 \\
	\end{bmatrix},\;\;  
	{\bs \delta}_4(t)={\bs \delta}_1(t)^T.$$ 
	%

	We denote the $i$th column of ${\bs \delta}_2(t)$ by $e_i$ with $e_{-i}=e_{13-i}$.  Note that hyperbolic pairs are $\{e_1,e_7\}$, $\{e_2,e_8\}$, $\{e_3,e_9\}$, $\{e_4,e_{10}\}$, $\{e_5,e_{11}\}$ and $\{e_6,e_{12}\}$. We calculate spinor coordinates using the Buchsbaum-Eisenbud map ${\bs a}_3(t)$. In Table \ref{tab:resn11} at the end of the paper, $\bar{i}$ denote the column corresponding to $e_{-i}$. 

	For $\mathcal{J}\subset [1,6]$, ${\bs a}(t)_{3,\mathcal{K}}=\widetilde{{\bs a}}(t)_{3,\mathcal{J}}^2$ where $\mathcal{K}=-\mathcal{J}\cup \mathcal{J}^c$ with cardinality of $\mathcal{J}$ is even. This means that there exists map $\widetilde{{\bs a}}(t)_3$ from $\widetilde S$ to half-spinor representation
	$V({\omega}_{6}, D_{6})$ sending 1 to the combination of basis vectors ${\bs v}_\mathcal{L}$ of weights
	$(\pm \frac{1}{2},\cdots,\pm \frac{1}{2})$ with spinor coordinates $\widetilde{{\bs a}}(t)_{3,\mathcal{L}}$ and even number of minuses (indicated by multi-index $\mathcal{L}$).
\end{proof}

We conclude with our main application. Using spinor coordinates, we show that the resolution (\ref{mapcone}), which is a generic doubling of perfect ideal with $5$ generators of Cohen-Macaulay type $2$, is not a specialization of the Kustin-Miller family given in Section~\ref{KMmodel}.

\begin{theorem} \label{notspofKM}
		The resolution given in (\ref{mapcone}) is not a specialization of the Kustin-Miller family in Section~\ref{KMmodel}.
\end{theorem}
\begin{proof}
Suppose resolution (\ref{mapcone}) is a specialization of KMM given in Section~\ref{KMmodel}. Specialization is a ring homorphism which takes $\bs a_{3,\mathcal{K}}$ for $\mathcal{K}\subset [1,6]$ of resolution (\ref{mapcone}) to $\bs a_{3,\mathcal{L}}$ for $\mathcal{L}\subset [1,6]$ of KMM. Therefore spinor coordiates in Table \ref{tab:resn11} goes to spinor coordinates in Table \ref{tab:7gens} of KMM. In Table \ref{tab:resn11}, we see that only one spinor coordinate is among  minimal generators of the ideal in resolution (\ref{mapcone}). Then by specialization  KMM can have at most one spinor coordinate among minimal generators of $I$ in Section~\ref{KMmodel}. This is not possible as  Table \ref{tab:7gens} has four spinor coordinates among minimal generators of $I$ in Section~\ref{KMmodel}. \end{proof}

\begin{remark} \label{finalrmk} Calculations in Examples \ref{spinci}, \ref{spinpfaff}, Proposition \ref{Kmspin}, and Theorem \ref{thspin5} show that at least one of the minimal generators of a Gorenstein ideal with 4, 6, or 7 generators are among the spinor coordinates. However, in case of a Gorenstein ideal with 9 generators, we see in Example~\ref{tomjerry} that none of the minimial generators of the ideal comes from spinor coordinates. This suggests that Gorenstein ideals of codimension 4 with up to 8 generators are easier to classify than those with more than 8 generators.
\end{remark}

\section{Acknowledgments}
J. Laxmi was supported by Fulbright-Nehru fellowship. J. Weyman was partially supported by NSF grant DMS 1802067, Sidney Professorial Fund, and Polish National Agency for Academic Exchange. E. Celikbas and J. Laxmi also acknowledge the partial support of the Sidney Professorial Fund from the University of Connecticut.

\newpage

\begin{figure}
	\caption{The Matrix ${\bs \psi}_2(t)$}
	\label{mat:psi_2(t)}
	\centering
	\scalebox{0.85}{  
		\rotatebox{90}{$\begin{bmatrix*}
			-x_{14}\tau_4-x_{24}\tau_3-x_{34}\tau_2 & x_{13}\tau_4+x_{23}\tau_3-x_{34}\tau_1 & -x_{12}\tau_4+x_{23}\tau_2+x_{24}\tau_1  & -x_{12}\tau_3-x_{13}\tau_2-x_{14}\tau_1 & \tau_5\\
			y_{14}\tau_4+y_{24}\tau_3+y_{34}\tau_2& -y_{13}\tau_4-y_{23}\tau_3+y_{34}\tau_1 & y_{12}\tau_4-y_{23}\tau_2-y_{24}\tau_1 &y_{12}\tau_3+y_{13}\tau_2+y_{14}\tau_1 & \tau_6\\
			x_{14}\tau_6+y_{14}\tau_5-\frac 12z_{124}\tau_3-\frac 12z_{134}\tau_2 &-x_{13}\tau_6-y_{13}\tau_5+\frac 12z_{123}\tau_3-\frac 12z_{134}\tau_1 & x_{12}\tau_6+y_{12}\tau_5+\frac 12z_{123}\tau_2+\frac 12z_{124}\tau_1 & 0 & 0 \\
			-x_{24}\tau_6-y_{24}\tau_5-\frac 12z_{124}\tau_4+\frac 12z_{234}\tau_2 & x_{23}\tau_6+y_{23}\tau_5+\frac 12z_{123}\tau_4+\frac 12z_{234}\tau_1 & 0 & -x_{12}\tau_6-y_{12}\tau_5-\frac 12z_{123}\tau_2-\frac 12z_{124}\tau_1 & 0\\
			x_{34}\tau_6+y_{34}\tau_5+\frac 12z_{134}\tau_4+\frac 12z_{234}\tau_3 & 0 & -x_{23}\tau_6-y_{23}\tau_5-\frac 12z_{123}\tau_4-\frac 12z_{234}\tau_1 & x_{13}\tau_6+y_{13}\tau_5-\frac 12z_{123}\tau_3+\frac 12z_{134}\tau_1 & 0\\
			0 & -x_{34}\tau_6-y_{34}\tau_5-\frac 12z_{134}\tau_4-\frac 12z_{234}\tau_3 &x_{24}\tau_6+y_{24}\tau_5+\frac 12z_{124}\tau_4-\frac 12z_{234}\tau_2 & -x_{14}\tau_6-y_{14}\tau_5+\frac 12z_{124}\tau_3+\frac 12z_{134}\tau_2 & 0 \\
			\end{bmatrix*}^T$}}
\end{figure}

		\begin{table}[ht]
			\caption{Spinor coordinates of resolution (\ref{mapcone})}
			\label{tab:resn11}
			{\renewcommand{\arraystretch}{1.33}
				\begin{tabular}{ |p{13cm}|}
					\hline
					\multicolumn{1}{|c|}{Cases for $\widetilde{a}(t)_{3,\mathcal{J}}$}\\
					\hline
													$\tilde{a}(t)_{3,\{\phi\}}=0$\\
								\hline
								$\tilde{a}(t)_{3,\{1,2\}}=iu(t)$\\
								\hline
								$\tilde{a}(t)_{3,\{1,3\}}=i(x_{24}u_{1,3,4}(t)-x_{34}u_{1,2,4}(t)-x_{14}u_{2,3,4}(t))$\\
								\hline
								$\tilde{a}(t)_{3,\{1,4\}}= i(x_{23}u_{1,3,4}(t)-x_{34}u_{1,2,3}(t)-x_{13}u_{2,3,4}(t))$\\
								\hline
								$\tilde{a}(t)_{3,\{1,5\}}=-x_{24}u_{1,2,3}(t)+x_{23}u_{1,2,4}(t)-x_{12}u_{2,3,4}(t)$\\
								\hline
								$\tilde{a}(t)_{3,\{1,6\}}=x_{13}u_{1,2,4}(t)-x_{14}u_{1,2,3}(t)-x_{12}u_{1,3,4}(t)$\\
								\hline 
								$\tilde{a}(t)_{3,\{2,3\}}=i(y_{34}u_{1,2,4}(t)-y_{24}u_{1,3,4}(t)+y_{14}u_{2,3,4}(t))$\\
								\hline
								$\tilde{a}(t)_{3,\{2,4\}}=i(y_{34}u_{1,2,3}(t)-y_{23}u_{1,3,4}(t)+y_{13}u_{2,3,4}(t))$\\
								\hline
								$\tilde{a}(t)_{3,\{2,5\}}=-y_{24}u_{1,2,3}(t)+y_{23}u_{1,3,4}(t)-y_{12}u_{2,3,4}(t)$\\
								\hline
								$\tilde{a}(t)_{3,\{2,6\}}=-y_{14}u_{1,2,3}(t)+y_{13}u_{1,2,4}(t)-y_{12}u_{1,3,4}(t)$\\
								\hline
								$\tilde{a}(t)_{3,\{3,4\}}=\frac{1}{2}[i(z_{2,3,4}u_{1,3,4}(t)-z_{1,3,4}u_{2,3,4}(t))]$\\
								\hline
								$\tilde{a}(t)_{3,\{3,5\}}=\frac{1}{2}[z_{2,3,4}u_{1,2,4}(t)-z_{1,2,4}u_{2,3,4}(t)]$\\
								\hline
								$\tilde{a}(t)_{3,\{3,6\}}=\frac{1}{2}[z_{1,3,4}u_{1,2,4}(t)-z_{1,2,4}u_{1,3,4}(t)]$\\
								\hline
								$\tilde{a}(t)_{3,\{4,5\}}=\frac{1}{2}[z_{2,3,4}u_{1,2,3}(t)-z_{1,2,3}u_{2,3,4}(t)]$\\
								\hline
								$\tilde{a}(t)_{3,\{4,6\}}=\frac{1}{2}[z_{1,3,4}u_{1,2,3}(t)-z_{1,2,3}u_{1,3,4}(t)]$\\
								\hline
								$\tilde{a}(t)_{3,\{5,6\}}=\frac{1}{2}[i(z_{1,2,4}u_{1,2,3}(t)-z_{1,2,3}u_{1,2,4}(t))]$\\
								\hline
								 $\tilde{a}(t)_{3,\{1,2,3,4\}}=\frac{1}{2}[i(\tau_4u_{1,3,4}(t)+\tau_3u_{2,3,4}(t)+2y_{34}f_1(t)+2x_{34}f_2(t))]$\\
								 \hline 
								 $\tilde{a}(t)_{3,\{1,2,3,5\}}=\frac{1}{4}[-\tau_4 u_{1,2,4}(t)+\tau_2u(t)_{2,3,4}-2y_{24}f_1(t)-2x_{24}f_2(t)]$\\
								 \hline
								 $\tilde{a}(t)_{3,\{1,2,3,6\}}=\frac{1}{2}[\tau_3 u_{1,2,4}(t)+\tau_2u_{1,3,4}(t)-2y_{14}f_1(t)-2x_{14}f_2(t)]$ \\
								 \hline
								 $\tilde{a}(t)_{3,\{1,2,4,5\}}=\frac{1}{4}[-\tau_4 u_{1,2,3}(t)-\tau_1 u_{2,3,4}(t)-2y_{23}f_1(t)-2x_{23}f_2(t)]$\\
								 \hline
								 $\tilde{a}(t)_{3,\{1,2,4,6\}}=\frac{1}{2}[\tau_3 u_{1,2,3}(t)-\tau_1u_{1,3,4}(t)-2y_{13}f_1(t)-2x_{13}f_2(t)]$\\
								 \hline
								 $\tilde{a}(t)_{3,\{1,2,5,6\}}=\frac{1}{4}[i(-\tau_2 u_{1,2,3}(t)-\tau_1u_{1,2,4}(t)-2y_{12}f_1(t)-2x_{12}f_2(t))]$\\
								 \hline
								  $\tilde{a}(t)_{3,\{1,3,4,5\}}=\frac{1}{2}[-\tau_5u_{2,3,4}(t)+z_{2,3,4}f_1(t)]$\\
								  \hline
								  $\tilde{a}(t)_{3,\{1,3,4,6\}}=\frac{1}{4}[i(-\tau_5u_{1,3,4}(t)+z_{1,3,4}f_1(t))]$\\
								  \hline
								  $\tilde{a}(t)_{3,\{2,3,4,5\}}=\frac{1}{2}[\tau_6u_{2,3,4}(t)-z_{2,3,4}f_2(t)]$\\
								  \hline
								  $\tilde{a}(t)_{3,\{2,3,4,6\}}=\frac{1}{2}[\tau_6u_{1,3,4}(t)-z_{1,3,4}f_2(t)]$\\
								  \hline
								  $\tilde{a}(t)_{3,\{2,3,5,6\}}=\frac{1}{2}[i(\tau_6u_{1,2,4}(t)-z_{1,2,4}f_2(t))]$\\
								  \hline
								  $\tilde{a}(t)_{3,\{2,4,5,6\}}=\frac{1}{2}[i(-\tau_6u_{1,2,3}(t)-z_{1,2,3}f_2(t))]$\\
								  \hline
								  $\tilde{a}(t)_{3,\{1,3,4,5\}}=\frac{1}{2}[i(-\tau_5u_{1,2,3}(t)+z_{1,2,3}f_1(t))]$\\
								  \hline 
								  $\tilde{a}(t)_{3,\{1,3,5,6\}}=\frac{1}{2}[i(-\tau_5u_{1,2,4}(t)+z_{1,2,4}f_1(t))]$\\
								  \hline
								  $\tilde{a}(t)_{3,\{1,4,5,6\}}=\frac{1}{2}[i(-\tau_5u_{1,2,3}(t)+z_{1,2,3}f_1(t))]$\\
								  \hline
								  $\tilde{a}(t)_{3,\{1,2,3,4,5,6\}}=\frac{1}{2}[i(\tau_6f_1(t)-\tau_5f_2(t))]$\\
								  \hline
				
			\end{tabular}}
		\end{table}
		\end{document}